\documentclass[12pt,a4paper,twoside]{article}

\usepackage{tikz,verbatim}

\title{\sc On an Extension\\
of Korn's First Inequality\\ 
to Incompatible Tensor Fields\\
on Domains of Arbitrary Dimensions}
\def\shorttitle{On an Extension of Korn's First Inequality to Incompatible Tensor Fields}
\def\pauthor{Patrizio Neff, Dirk Pauly, Karl-Josef Witsch}

\def\mylabelonoff{off}
\def\allowdisbrk{no}

\usepackage{a4,exscale,ifthen,amsfonts,amssymb,amsmath,amscd,graphicx}
\usepackage[english]{babel}
\usepackage[mathscr]{eucal}

\author{{\sf\pauthor}}
\numberwithin{equation}{section}

\newenvironment{acknow}{{\vspace*{1cm}\noindent\bf Acknowledgements }}{}
\newcommand{\bewboxw}{\mbox{}\hfill $\square$ \\}

\newenvironment{proof}{{\noindent\bf Proof }}{\bewboxw}
\newenvironment{proofof}[1]{{\noindent\bf Proof of #1 }}{\bewboxw}

\newcommand{\keywords}[1]{{\noindent\bf Key Words }#1}

\ifthenelse{\equal{\mylabelonoff}{on}}
{\newcommand{\mylabel}[1]{\label{#1}\fbox{{\rm #1}}}}{\newcommand{\mylabel}[1]{\label{#1}\makebox[0mm][]{}}}
\ifthenelse{\equal{\allowdisbrk}{yes}}
{\allowdisplaybreaks}{}

\usepackage[mathscr]{eucal}
\usepackage{color}
\usepackage[all]{xy}

\def\rc{\color{red}}

\newcommand{\ds}{\displaystyle}
\newcommand{\ol}{\overline}
\newcommand{\ul}{\underline}

\newcommand{\nz}{\mathbb{N}}

\newcommand{\rz}{\mathbb{R}}
\newcommand{\rzp}{\rz_{+}}

\newcommand{\rt}{\rz^3}

\newcommand{\rN}{\rz^N}
\newcommand{\rNtN}{\rz^{N\times N}}

\DeclareMathOperator{\p}{\partial}

\newcommand{\na}{\nabla}

\DeclareMathOperator{\ed}{d}

\DeclareMathOperator{\cd}{\delta}

\DeclareMathOperator{\grad}{grad}
\DeclareMathOperator{\Grad}{Grad}

\DeclareMathOperator{\curl}{curl}
\DeclareMathOperator{\Curl}{Curl}

\renewcommand{\div}{\operatorname{div}}
\DeclareMathOperator{\Div}{Div}

\newcommand{\trans}[1]{{#1}^\top}
\newcommand{\T}{T}
\newcommand{\TS}{S}
\newcommand{\TR}{R}

\newcommand{\TP}{P}

\newcommand{\Tskew}{S}

\newcommand{\om}{\Omega}
\newcommand{\dom}{\p\!\om}

\newcommand{\ga}{\Gamma}
\newcommand{\gat}{\ga_{t}}
\newcommand{\gan}{\ga_{n}}
\newcommand{\gatj}{\ga_{t,j}}

\DeclareMathOperator{\supp}{supp}
\DeclareMathOperator{\sym}{sym}
\DeclareMathOperator{\dist}{dist}
\DeclareMathOperator{\interior}{int}
\DeclareMathOperator{\tr}{tr}
\renewcommand{\skew}{\operatorname{skew}}

\newcommand{\dvec}[3]{\begin{bmatrix}#1\\#2\\#3\end{bmatrix}}

\DeclareMathOperator{\id}{id}
\newcommand{\dl}{\,d\lambda}

\DeclareMathOperator{\so}{\mathfrak{so}}

\def\set#1#2{\{#1\,:\,#2\}}

\newcommand{\cp}{c_{\mathtt{p}}}
\newcommand{\cpq}{c_{\mathtt{p},q}}
\newcommand{\cpqmo}{c_{\mathtt{p},q-1}}
\newcommand{\cpqpo}{c_{\mathtt{p},q+1}}
\newcommand{\cpz}{c_{\mathtt{p},0}}
\newcommand{\cpo}{c_{\mathtt{p},1}}
\newcommand{\cm}{c_{\mathtt{m}}}

\newcommand{\cpf}[1]{c_{\mathtt{pf},#1}}
\newcommand{\ck}{c_{\mathtt{k}}}

\newcommand{\ckt}{c_{\mathtt{k},t}}

\newcommand{\cks}{c_{\mathtt{k},s}}

\newcommand{\cktj}{c_{\mathtt{k},t,j}}
\newcommand{\cksj}{c_{\mathtt{k},s,j}}

\DeclareMathOperator{\Lebesgue}{\mathsf{L}}
\newcommand{\Lgen}[2]{\Lebesgue^{#1}_{#2}}

\def\Lt{\Lgen{2}{}}
\def\Ltom{\Lt(\om)}

\newcommand{\qLt}[1]{\Lgen{2,#1}{}}
\newcommand{\Ltq}{\qLt{q}}
\newcommand{\Ltqpo}{\qLt{q+1}}
\newcommand{\Ltqmo}{\qLt{q-1}}
\newcommand{\Ltqom}{\Ltq(\om)}
\newcommand{\Ltqpoom}{\Ltqpo(\om)}
\newcommand{\Ltqmoom}{\Ltqmo(\om)}

\DeclareMathOperator{\DSobolev}{\mathsf{D}}
\newcommand{\Dgen}[3]{\overset{#3}{\DSobolev}{}^{#1}_{#2}}
\newcommand{\qD}[1]{\Dgen{#1}{}{}}

\newcommand{\qDc}[1]{\Dgen{#1}{}{\circ}}
\newcommand{\qDcz}[1]{\Dgen{#1}{0}{\circ}}
\newcommand{\Dq}{\qD{q}}

\newcommand{\Dqmo}{\qD{q-1}}

\newcommand{\Dqc}{\qDc{q}}

\newcommand{\Dqmoc}{\qDc{q-1}}
\newcommand{\Dqcz}{\qDcz{q}}

\newcommand{\Dqom}{\Dq(\om)}

\newcommand{\Dqmoom}{\Dqmo(\om)}

\newcommand{\Dqcom}{\Dqc(\om)}

\newcommand{\Dqcgatom}{\Dqc(\gat,\om)}
\newcommand{\Dqczgatom}{\Dqcz(\gat,\om)}
\newcommand{\Dqmocgatom}{\Dqmoc(\gat,\om)}

\newcommand{\Dqczganom}{\Dqcz(\gan,\om)}
\newcommand{\Dqmocganom}{\Dqmoc(\gan,\om)}

\DeclareMathOperator{\DeSobolev}{\Delta}
\newcommand{\Degen}[3]{\overset{#3}{\DeSobolev}{}^{#1}_{#2}}
\newcommand{\qDe}[1]{\Degen{#1}{}{}}
\newcommand{\qDec}[1]{\Degen{#1}{}{\circ}}
\newcommand{\qDez}[1]{\Degen{#1}{0}{}}
\newcommand{\qDecz}[1]{\Degen{#1}{0}{\circ}}
\newcommand{\Deq}{\qDe{q}}
\newcommand{\Deqpo}{\qDe{q+1}}

\newcommand{\Deqc}{\qDec{q}}
\newcommand{\Deqpoc}{\qDec{q+1}}

\newcommand{\Deqz}{\qDez{q}}

\newcommand{\Deqcz}{\qDecz{q}}

\newcommand{\Deqom}{\Deq(\om)}
\newcommand{\Deqpoom}{\Deqpo(\om)}

\newcommand{\Deqzom}{\Deqz(\om)}

\newcommand{\Deqcganom}{\Deqc(\gan,\om)}
\newcommand{\Deqpocganom}{\Deqpoc(\gan,\om)}
\newcommand{\Deqczganom}{\Deqcz(\gan,\om)}

\newcommand{\Deqczgatom}{\Deqcz(\gat,\om)}
\newcommand{\Deqpocgatom}{\Deqpoc(\gat,\om)}

\DeclareMathOperator{\Sobolev}{\mathsf{H}}
\newcommand{\Hgen}[3]{\overset{#3}{\Sobolev}{}^{#1}_{#2}}
\def\Ho{\Hgen{1}{}{}}

\def\Hk{\Hgen{k}{}{}}
\def\Hoom{\Ho(\om)}

\def\Hkom{\Hk(\om)}
\def\Hoot{\Hgen{1/2}{}{}}
\def\Hootom{\Hoot(\om)}

\def\Hoc{\Hgen{1}{}{\circ}}
\def\Hocom{\Hoc(\om)}
\def\Hocgatom{\Hoc(\gat;\om)}

\DeclareMathOperator{\Cont}{\mathsf{C}}
\newcommand{\Cgen}[2]{\overset{#2}{\Cont}{}^{#1}}
\def\Cz{\Cgen{0}{}}

\def\Czomb{\Cz(\ol{\om})}
\def\Ci{\Cgen{\infty}{}}
\def\Cic{\Cgen{\infty}{\circ}}
\def\Ciq{\Cgen{\infty,q}{}}
\def\Ciqc{\Cgen{\infty,q}{\circ}}

\def\Ciom{\Ci(\om)}
\def\Cicom{\Cic(\om)}

\def\Ciqcom{\Ciqc(\om)}

\def\Ciqcgatom{\Ciqc(\gat,\om)}

\DeclareMathOperator{\dirichlet}{\mathcal{H}}
\newcommand{\qharmdi}[2]{\dirichlet^{#1}_{#2}(\om)}
\newcommand{\harmdi}{\qharmdi{}{}}
\newcommand{\harmdiz}{\qharmdi{0}{}}
\newcommand{\harmdio}{\qharmdi{1}{}}
\newcommand{\harmdiq}{\qharmdi{q}{}}

\newcommand{\Hggen}[3]{\overset{#2}{\Sobolev}(\grad;#3)}
\newcommand{\HGgen}[3]{\overset{#2}{\Sobolev}(\Grad;#3)}
\newcommand{\Hcgen}[3]{\overset{#2}{\Sobolev}(\curl_{#1};#3)}
\newcommand{\HCgen}[3]{\overset{#2}{\Sobolev}(\Curl_{#1};#3)}
\newcommand{\Hdgen}[3]{\overset{#2}{\Sobolev}(\div_{#1};#3)}
\newcommand{\HDgen}[3]{\overset{#2}{\Sobolev}(\Div_{#1};#3)}

\newcommand{\Hgom}{\Hggen{}{}{\om}}
\newcommand{\HGom}{\HGgen{}{}{\om}}
\newcommand{\Hcom}{\Hcgen{}{}{\om}}
\newcommand{\HCom}{\HCgen{}{}{\om}}
\newcommand{\Hdom}{\Hdgen{}{}{\om}}
\newcommand{\HDom}{\HDgen{}{}{\om}}
\newcommand{\Hgcom}{\Hggen{}{\circ}{\om}}

\newcommand{\Hccom}{\Hcgen{}{\circ}{\om}}

\newcommand{\Hdcom}{\Hdgen{}{\circ}{\om}}

\newcommand{\HCzom}{\HCgen{0}{}{\om}}
\newcommand{\Hdzom}{\Hdgen{0}{}{\om}}
\newcommand{\HDzom}{\HDgen{0}{}{\om}}

\newcommand{\HDczom}{\HDgen{0}{\circ}{\om}}
\newcommand{\Hgcgatom}{\Hggen{}{\circ}{\gat,\om}}
\newcommand{\Hgcganom}{\Hggen{}{\circ}{\gan,\om}}
\newcommand{\HGcgatom}{\HGgen{}{\circ}{\gat,\om}}

\newcommand{\Hccgatom}{\Hcgen{}{\circ}{\gat,\om}}
\newcommand{\HCcgatom}{\HCgen{}{\circ}{\gat,\om}}
\newcommand{\Hdcgatom}{\Hdgen{}{\circ}{\gat,\om}}
\newcommand{\Hdcganom}{\Hdgen{}{\circ}{\gan,\om}}
\newcommand{\HDcganom}{\HDgen{}{\circ}{\gan,\om}}
\newcommand{\Hcczgatom}{\Hcgen{0}{\circ}{\gat,\om}}
\newcommand{\HCczgatom}{\HCgen{0}{\circ}{\gat,\om}}
\newcommand{\Hdczganom}{\Hdgen{0}{\circ}{\gan,\om}}
\newcommand{\HDczganom}{\HDgen{0}{\circ}{\gan,\om}}
\newcommand{\Hccganom}{\Hcgen{}{\circ}{\gan,\om}}
\newcommand{\HCcganom}{\HCgen{}{\circ}{\gan,\om}}

\DeclareMathOperator{\XSobolev}{\mathsf{X}}
\DeclareMathOperator{\YSobolev}{\mathsf{Y}}

\newcommand{\xqmoom}{\XSobolev^{q-1}(\om)}

\newcommand{\yqpoom}{\YSobolev^{q+1}(\om)}

\newcommand{\normdst}{\hspace{-0.4ex}}
\newcommand{\scp}[2]{\left\langle#1,#2\right\rangle}
\newcommand{\scps}[2]{\langle#1,#2\rangle}

\newcommand{\scpLtom}[2]{\scp{#1}{#2}_{\Ltom}}

\newcommand{\scpLtqom}[2]{\scp{#1}{#2}_{\Ltqom}}

\newcommand{\norm}[1]{\left|\normdst\left|#1\right|\normdst\right|}

\newcommand{\normLtom}[1]{\norm{#1}_{\Ltom}}

\newcommand{\normLtqom}[1]{\norm{#1}_{\Ltqom}}

\newcommand{\normLtqpoom}[1]{\norm{#1}_{\Ltqpoom}}

\newcommand{\normLtqmoom}[1]{\norm{#1}_{\Ltqmoom}}

\newtheorem{lem}{Lemma}
\newtheorem{defi}[lem]{Definition}
\newtheorem{theo}[lem]{Theorem}
\newtheorem{cor}[lem]{Corollary}
\newtheorem{rem}[lem]{Remark}

\renewcommand{\Tskew}{A}
\newcommand{\Th}{\hat{\T}}
\newcommand{\TPh}{\hat{\TP}}
\newcommand{\MCP}{{\sf MCP}}
\newcommand{\MAP}{{\sf MAP}}
\newcommand{\KP}{{\sf KP}}
\newcommand{\pair}{(\om,\gat)}
\newcommand{\pairgan}{(\om,\gan)}
\newcommand{\calE}{\mathcal{E}}
\newcommand{\calEt}{\tilde{\mathcal{E}}}
\newcommand{\calS}{\mathcal{S}}

\newcommand{\cone}{c_{1}}
\newcommand{\ctwo}{c_{2}}
\newcommand{\soN}{\so(N)}

\begin{document}

\maketitle{}


\thispagestyle{empty}

\begin{abstract}
For a bounded domain $\om$ in $\rN$ with Lipschitz boundary $\ga=\dom$
and a relatively open and non-empty `admissible' subset $\gat$ of $\ga$,
we prove the existence of a positive constant $c$ such that inequality
\begin{align}
\mylabel{kornineqabstract}
c\norm{\T}_{\Lt(\om,\rNtN)}
\leq\norm{\sym\T}_{\Lt(\om,\rNtN)}
+\norm{\Curl\T}_{\Lt(\om,\rz^{N\times N(N-1)/2})}
\end{align}
holds for all tensor fields $\T\in\HCgen{}{\circ}{\gat,\om,\rNtN}$,
this is, for all square-integrable tensor fields $\T:\om\to\rNtN$
having a row-wise square-integrable rotation tensor field
$\Curl\T:\om\to\rz^{N\times N(N-1)/2}$
and vanishing row-wise tangential trace on $\gat$. 

For compatible tensor fields $\T=\na v$ with $v\in\Ho(\om,\rN)$
having vanishing tangential Neumann trace on $\gat$
the inequality \eqref{kornineqabstract} reduces 
to a non-standard variant of Korn's first inequality
since $\Curl\T=0$, while for skew-symmetric tensor fields $\T$
Poincar\'e's inequality is recovered.

If $\gat=\emptyset$, our estimate \eqref{kornineqabstract} 
still holds at least for simply connected $\om$
and for all tensor fields $\T\in\HCgen{}{}{\om,\rNtN}$
which are $\Lgen{2}{}(\om,\rNtN)$-perpendicular to $\soN$, i.e.,
to all skew-symmetric constant tensors.\\
\end{abstract}

\tableofcontents

\section{Introduction and Main Results}
\mylabel{secmainresults}

We extend the Korn-type inequalities from 
\cite{neffpaulywitschgenkornrtsli}
presented earlier in less general settings in
\cite{neffpaulywitschgenkornpamm,neffpaulywitschgenkornrt,
neffpaulywitschgenkornrtzap,neffpaulywitschgenkornrn}
to the $N$-dimensional case.
For this, let $N\in\nz$ and $\om$ be a bounded domain in $\rN$ 
as well as $\gat$ be an open subset of its boundary $\ga:=\dom$.
Our main result reads:

\begin{theo}[Main Theorem]
\mylabel{maintheo}
Let the pair $\pair$ be admissible\footnote{The 
precise meaning of `admissible' will be given 
in Definition \ref{domdef}.}.
There exist constants $0<\cone\leq\ctwo$ such that the following estimates hold: 
\begin{itemize}
\item[\bf(i)] 
If $\gat\neq\emptyset$, then the inequality
$$\normLtom{\T}\leq\cone\big(\normLtom{\sym\T}+\normLtom{\Curl\T}\big)$$
holds for all tensor fields $\T\in\HCcgatom$. 
In other words, on $\HCcgatom$ the right hand side 
defines a norm equivalent to the standard norm in $\HCom$.
\item[\bf(ii)] 
If $\gat=\emptyset$, then for all tensor fields $\T\in\HCom$\footnote{If 
$\gat=\emptyset$, then $\HCcgatom=\HCom$.} 
there exists a piece-wise constant skew-symmetric tensor field $\Tskew$, such that
$$\normLtom{\T-\Tskew}\leq\ctwo\big(\normLtom{\sym\T}+\normLtom{\Curl\T}\big).$$
Note that in general $\Tskew\notin\HCom$.
\item[\bf(ii')] 
If $\gat=\emptyset$ and $\om$ is additionally simply connected, 
then for all tensor fields $\T$ in $\HCom$ 
there exists a uniquely determined constant 
skew-symmetric tensor field $\Tskew=\Tskew_{\T}\in\soN$\footnote{$\soN$ 
denotes the set of all constant skew-symmetric tensors, i.e., 
$(N\times N)$-matrices.}, such that
$$\normLtom{\T-\Tskew_{\T}}\leq\cone\big(\normLtom{\sym\T}+\normLtom{\Curl\T}\big).$$
Since $\Tskew_{\T}\in\HCzom$ one can easily estimate 
$\norm{\T-\Tskew_{\T}}_{\HCom}$ as well.
Moreover, $\T-\Tskew_{\T}\in\HCom\cap\soN^{\bot}$ and
$\Tskew_{\T}=0$ if and only if $\T\bot\soN$. 
Thus, the inequality in (i) holds for all $\T\in\HCom\cap\soN^{\bot}$ as well.
Therefore, also on $\HCom\cap\soN^{\bot}$ the right hand side 
defines a norm equivalent to the standard norm in $\HCom$.
\end{itemize}
\end{theo}

\begin{rem}
\mylabel{maintheorem}
\begin{itemize}
\item[\bf(i)] 
Here, the differential operator $\Curl$ 
denotes the row-wise application
of the standard $\curl$ in $\rN$ and a tensor field $\T$
belongs to the Hilbert Sobolev-type space $\HCcgatom$
if $\T$ and its distributional $\Curl\T$ belong both 
to the standard Lebesgue spaces $\Ltom$
and the row-wise weak tangential trace of $\T$ 
vanishes at the boundary part $\gat$.
Exact definitions of all spaces and operators used
will be given in section \ref{secdef}.
\item[\bf(ii)] 
In (ii') the special constant skew-symmetric tensor field $\Tskew_{\T}$
is given explicitly by $\Tskew_{\T}=\pi_{\soN}\T\in\soN$,
where $\pi_{\soN}:\Ltom\to\soN$ denotes the $\Ltom$-orthogonal projection onto $\soN$
and can be represented by
\begin{align*}
\pi_{\soN}\T&=\skew\oint_{\om}\T\dl\in\soN.
\intertext{Furthermore, $\Tskew_{\T}$ can also be computed by}
\Tskew_{\T}=\Tskew_{\TR}:=\pi_{\soN}\TR&=\skew\oint_{\om}\TR\dl\in\soN,
\end{align*}
where $\TR$ denotes the Helmholtz projection of $\T$ 
onto $\HCzom$ according to Corollary \ref{helmdecoten}.
\item[\bf(iii)] 
The constants $\cone$ and $\ctwo$ are given by
\eqref{explicitconst} and \eqref{explicitconsttwo}
and these depend in a simply algebraic way only on the constants $\ck,\cm$
in Korn's first and the Maxwell inequality.
\end{itemize}
\end{rem}

For the proof of Theorem \ref{maintheo}
we follow in close lines the proofs from \cite{neffpaulywitschgenkornrtsli}.
Therefore, again we need to combine three crucial tools, namely
\begin{itemize}
\item a Maxwell estimate, Corollary \ref{poincaremaxten};
\item a Helmholtz decomposition, Corollary \ref{helmdecoten};
\item a generalized version of Korn's first inequality, Lemma \ref{genkornlem}.
\end{itemize}
Our assumptions on the domain $\om$ and the part of the boundary $\gat$, i.e.,
on the pair $\pair$, are precisely made for this three major tools to hold.
We will present these assumptions in section \ref{secdef}
and a pair $\pair$ satisfying those will be called admissible.

Theorem \ref{maintheo} can be looked at as a common 
generalization and formulation of two well known and very important 
classical inequalities, namely  {\bf Korn's first} and {\bf Poincar\'e's inequality}.
This is, taking irrotational tensor fields $\T$, i.e.,
$\Curl\T=0$, then a non-standard version of Korn's first inequality
\begin{align*}
c\normLtom{\T-\Tskew_{\T}}&\leq\normLtom{\sym\T}
\intertext{holds for all $\T\in\HCczgatom$, 
where $\Tskew_{\T}=0$ if $\gat\neq\emptyset$.
Another, less general choice, is $\T=\na v$ yielding}
c\normLtom{\na v-\Tskew_{\na v}}&\leq\normLtom{\sym\na v}
\end{align*}
with some vector field $v$ belonging to $\Hocgatom$
or just to $\Hoom$ with $\na v_{n}$, $n=1,\ldots,N$,
normal at $\gat$. Note that
$$\na\Hocgatom,\,\na\set{v\in\Hoom}{\na v_{n}\text{ normal at }\gat\,\forall\,n=1,\ldots,N}
\subset\HCczgatom.$$
On the other hand, taking a skew-symmetric tensor field $\T$, i.e.,
$\sym\T=0$, then Poincar\'e's inequality in disguise
\begin{align*}
c\normLtom{\T-\Tskew_{\T}}&\leq\normLtom{\Curl\T}
\intertext{appears, where again $\Tskew_{\T}=0$ if $\gat\neq\emptyset$. 
We note that since $\T$ can be identified with 
a vector field $v$ and the $\Curl$ 
is as good as the gradient $\na$ on $v$ we have}
c\normLtom{v-c_{v}}&\leq\normLtom{\na v}
\end{align*}
with $c_{v}\in\rN$ and $c_{v}=0$ if $\gat\neq\emptyset$. 
These connections between Korn's first and Poincar\'e's inequalities
and also to the Maxwell inequalities 
and the more general Poincar\'e-type inequalities
are illustrated in Figure \ref{diagram}.

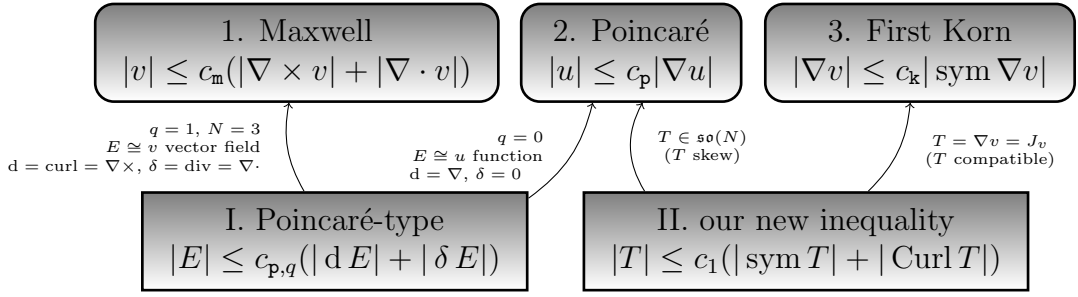
\begin{figure}
\begin{tikzpicture}
\node[rectangle,draw,line width=1pt,rounded corners=7pt,shading=axis] (m) at (0,2.5)
{\begin{tabular}{c}
1. Maxwell\\
$|v|\leq\cm(|\na\times v|+|\na\cdot v|)$
\end{tabular}};
\node[rectangle,draw,line width=1pt,rounded corners=7pt,shading=axis] (p) at (4.4,2.5)
{\begin{tabular}{c}
2. Poincar\'e\\
$|u|\leq\cp|\na u|$
\end{tabular}};
\node[rectangle,draw,line width=1pt,rounded corners=7pt,shading=axis] (k) at (8.2,2.5)
{\begin{tabular}{c}
3. First Korn\\
$|\na v|\leq\ck|\sym\na v|$
\end{tabular}};
\node[rectangle,draw,line width=1pt,rounded corners=0pt,shading=axis] (gp) at (0.5,0)
{\begin{tabular}{c}
I. Poincar\'e-type\\
$|E|\leq\cpq(|\ed E|+|\cd E|)$
\end{tabular}};
\node[rectangle,draw,line width=1pt,rounded corners=0pt,shading=axis] (o) at (6.7,0)
{\begin{tabular}{c}
II. our new inequality\\
$|\T|\leq\cone(|\sym\T|+|\Curl\T|)$
\end{tabular}};
\path[->] 
(gp) edge [bend left=20] node [left] 
{\tiny\begin{tabular}{r}
$q=1$, $N=3$\\
$E\cong v$ vector field\\
$\ed=\curl=\na\times$, $\cd=\div=\na\cdot$
\end{tabular}} (m)
edge [bend right=20] node [left] 
{\tiny\begin{tabular}{r}
$q=0$\\
$E\cong u$ function\\
$\ed=\na$, $\cd=0$ \hspace*{2mm}
\end{tabular}} (p)
(o) edge [bend left=30] node [right] 
{\tiny\begin{tabular}{c}
$\T\in\soN$\\
($\T$ skew)
\end{tabular}} (p)
edge [bend right=20] node [right] 
{\tiny\begin{tabular}{c}
$\T=\na v=J_{v}$\\
($\T$ compatible)
\end{tabular}} (k);
\end{tikzpicture}
\caption{The three fundamental inequalities are implied by two. 
For the constants we have $\cp=\cpz$, $\cm=\cpo$ and $\ck,\cp\leq\cone$.}
\mylabel{diagram}
\end{figure}

\section{Definitions and Preliminaries}
\mylabel{secdef}

As mentioned before,
let generally $N\in\nz$ and $\om$ be a bounded domain in $\rN$ 
as well as $\gat$ be an open subset of the boundary $\ga=\dom$.
We will use the notations from our earlier papers
\cite{neffpaulywitschgenkornrn}
and
\cite{neffpaulywitschgenkornpamm,neffpaulywitschgenkornrt,
neffpaulywitschgenkornrtzap,neffpaulywitschgenkornrtsli}.

\subsection{Differential Forms}

In particular, we denote the Lebesgue spaces
of square-integrable $q$-forms\footnote{alternating 
differential forms of rank $q\in\{0,\dots,N\}$}
by $\Ltqom$.
Moreover, we have the standard Sobolev-type spaces 
for the exterior derivative $\ed$ and co-derivative 
$\cd:=(-1)^{(q-1)N}*\ed*$ (acting on $q$-forms)
\begin{align*}
\Dqom&:=\set{E\in\Ltqom}{\ed E\in\Ltqpoom},\\
\Deqom&:=\set{H\in\Ltqom}{\cd H\in\Ltqmoom},
\intertext{where as usual $*$ denotes Hodge's star isomorphism.
$\Ciqcom$ is the space of smooth and compactly supported 
$q$-forms on $\om$, often called test space.
Due to the more complex geometry and topology of the domain $\om$
and its boundary parts $\ga,\gat$
we need some more test spaces}
\Ciqcgatom&:=\set{E\in\Ciq(\ol{\om})}{\dist(\supp E,\gat)>0},\\
\Ciq(\ol{\om})&:=\set{E|_{\om}}{E\in\Ciqc(\rN)}.
\intertext{Then, we define}
\Dqcgatom&:=\ol{\Ciqcgatom}
\end{align*}
taking the closure in $\Dqom$ 
and note that a $q$-form in $\Dqcgatom$ has
generalized vanishing tangential trace\footnote{This 
can be seen easily by Stokes' theorem.} on $\gat$.
If $\gat=\ga$ we can identify $\Ciqcgatom$ with $\Ciqcom$ and write
\begin{align*}
\Dqcgatom&\,=\ol{\Ciqcgatom}=\ol{\Ciqcom}=:\Dqcom.
\intertext{If $\gat=\emptyset$ we have $\Ciqcgatom=\Ciq(\ol{\om})$ and thus}
\Dqcgatom&\,=\ol{\Ciqcgatom}=\ol{\Ciq(\ol{\om})}\subset\Dqom.
\end{align*}
Equality in the last relation means 
the density result $\ol{\Ciq(\ol{\om})}=\Dqom$,
which holds, e.g., if $\om$ has 
the segment property\footnote{See, e.g., 
\cite{agmonbook,wlokabook,kuhndiss}.}.
The latter is valid, e.g., for domains with Lipschitz boundary. 
An index $0$ at the lower right corner
indicates vanishing derivatives, e.g.,
\begin{align*}
\Dqczgatom&:=\set{E\in\Dqcgatom}{\ed E=0}.
\end{align*}
Analogously, we introduce the corresponding Sobolev-type spaces 
for the co-derivative $\cd$ which are usually assigned to the
boundary complement $\gan:=\ga\setminus\ol{\gat}$ of $\gat$.
We have, e.g.,
$$\Deqzom=\set{H\in\Deqom}{\cd H=0},\quad
\Deqcganom,\quad\Deqczganom,$$
where in the latter spaces 
a vanishing normal trace on $\gan$ is generalized.
Moreover, we define the spaces of so called 
`harmonic Dirichlet-Neumann forms'
\begin{align}
\mylabel{dirichletdef}
\harmdiq:=\Dqczgatom\cap\Deqczganom.
\end{align}
We note that in classical terms a harmonic Dirichlet-Neumann 
$q$-form $E$ satisfies
$$\ed E=0,\quad\cd E=0,\quad
\iota^{*}E|_{\gat}=0,\quad\iota^{*}*E|_{\gan}=0,$$
where $\iota^{*}$ denotes the pullback 
of the canonical embedding $\iota:\ga\hookrightarrow\ol{\om}$
and the restrictions to $\gat$ and $\gan$ should be understood as pullbacks as well.
Equipped with their natural graph norms all these spaces are Hilbert spaces.

Now, we can begin to introduce our regularity assumptions on the boundary $\ga$
and the interface $\gamma:=\ol{\gat}\cap\ol{\gan}$. We start with the following:

\begin{defi}
\mylabel{defmaxcomp}
The pair $\pair$ has the `Maxwell compactness property' {\sf(MCP)}, 
if for all $q$ the embeddings
$$\Dqcgatom\cap\Deqcganom\hookrightarrow\Ltom$$
are compact.
\end{defi}

\begin{rem}
\mylabel{remmaxcomp}
\begin{itemize}
\item[\bf(i)] 
There exists a substantial amount of literature and different proofs for the \MCP.
See for example the papers and books of
Costabel, Kuhn, Leis, Pauly, Picard, Saranen, Weber, Weck, Witsch
\cite{costabelremmaxlip,kuhnpaulyregmax,leistheoem,leistheomaxgmd,leisbook,paulytimeharm,paulystatic,
paulyasym,paulydeco,picardpotential,picardboundaryelectro,picardcomimb,picardlowfreqmax,picarddeco,
picardweckwitschxmas,saranenmaxkegel,saranenmaxnichtglatt,webercompmax,weckmax,witschremmax}. 
All these papers are concerned with the special cases $\gat=\ga$ resp. $\gat=\emptyset$.
For the case $N=3$, $q=1$, i.e., $\om\subset\rt$, we refer to
\cite{costabelremmaxlip,leistheoem,leistheomaxgmd,leisbook,picardboundaryelectro,picardlowfreqmax,
picardweckwitschxmas,saranenmaxkegel,saranenmaxnichtglatt,webercompmax,witschremmax},
whereas for the general case, i.e., $\om\subset\rN$ 
or even $\om$ a Riemannian manifold, we correspond to
\cite{kuhnpaulyregmax,paulytimeharm,paulystatic,paulyasym,paulydeco,picardpotential,
picardcomimb,picarddeco,weckmax}. We note that even weaker regularity of $\ga$ than Lipschitz 
is sufficient for the {\MCP} to hold.
The first proof of the {\MCP} for non-smooth domains and
even for smooth Riemannian manifolds with non-smooth boundaries (cone property)
was given in 1974 by Weck in \cite{weckmax}.
To the best of our knowledge,
the strongest result so far can be found in
the paper of Picard, Weck and Witsch \cite{picardweckwitschxmas}.
See also our discussion in \cite{neffpaulywitschgenkornrtsli}.
An interesting proof has been given by Costabel in \cite{costabelremmaxlip}.
He made the detour of showing more fractional Sobolev regularity for the vector fields.
More precisely, he was able to prove
that for Lipschitz domains $\om\subset\rt$ and $q=1$ the embedding
\begin{align}
\mylabel{contembhoh}
\Dqcom\cap\Deqom&\hookrightarrow\Hootom
\intertext{is continuous. Then, for all $0\leq k<1/2$ the embeddings}
\Dqcom\cap\Deqom&\hookrightarrow\Hkom\nonumber
\end{align}
are compact, especially for $k=0$, where $\Hkom=\Ltom$ holds.
\item[\bf(ii)] 
For the general case $\emptyset\subset\gat\subset\ga$
with possibly $\emptyset\subsetneq\gat\subsetneq\ga$,
Jochmann gave a proof for the {\MCP} in \cite{jochmanncompembmaxmixbc},
where he considered the special case of a bounded domain $\om\subset\rt$.
He can admit $\om$ to be Lipschitz and $\gamma$ to be a Lipschitz interface.
Generalizing the ideas of Weck in \cite{weckmax},
Kuhn showed in his dissertation \cite{kuhndiss} that the {\MCP} holds for
domains $\om\subset\rN$ or even for smooth Riemannian manifolds $\om$
with `smooth' boundary and admissible interface $\gamma$. 
See also our discussion in \cite{neffpaulywitschgenkornrtsli}.
A result, which meets our needs, has been proved quite recently 
by M. Mitrea and his collaborators.
More precisely, we will use results by Gol'dshtein and Mitrea (I. \& M.) from
\cite{goldshteinmitreairinamariushodgedecomixedbc}.
In the language of this paper
we assume $\om$ to be a weakly Lipschitz domain,
this is, $\om$ is a Lipschitz manifold with boundary, 
see \cite[Definition 3.6]{goldshteinmitreairinamariushodgedecomixedbc},
and $\gat\subset\ga$ to be an admissible patch
(yielding $\gamma$ to be an admissible path), i.e.,
$\gat$ is a Lipschitz submanifold with boundary, 
see \cite[Definition 3.7]{goldshteinmitreairinamariushodgedecomixedbc}.
Roughly speaking, $\om$ and $\gat$ are defined by Lipschitz functions.
Here, the main point in proving the {\MCP}, i.e.,
\cite[Proposition 4.4, (4.21)]{goldshteinmitreairinamariushodgedecomixedbc},
is that then $\om$ is locally Lipschitz diffeomorphic to
a `creased domain' in $\rN$, first introduced by Brown
in \cite{brownmixedproblaplacelipschitz}.
See \cite[Section 3.6]{goldshteinmitreairinamariushodgedecomixedbc} 
for more details and to find the statement
`Informally speaking, the pieces in which the boundary is partitioned
are admissible patches which meet at an angle $<\pi$.
In particular, creased domains are inherently non-smooth'.
Whereas in \cite{goldshteinmitreairinamariushodgedecomixedbc} 
everything is defined in the more general framework of manifolds,
in \cite{jakabmitreairinamariusfinensolhodgedeco}
the {\MCP} is proved 
by Jakab and Mitrea (I. \& M.)
for creased domains $\om\subset\rN$.
By the Lipschitz diffeomorphisms, the {\MCP} holds then
for general manifolds/domains $\om$ as well.
In \cite{jakabmitreairinamariusfinensolhodgedeco} the authors
follow and generalize the idea \eqref{contembhoh} 
of Costabel from \cite{costabelremmaxlip}.
Particularly, in \cite[(1.2), Theorem 1.1, (1.9)]{jakabmitreairinamariusfinensolhodgedeco}
the following regularity result has been proved:
For all $q$ the embeddings
\begin{align*}
\Dqcgatom\cap\Deqcganom&\hookrightarrow\Hootom
\intertext{are continuous. Therefore, as before,
for all $q$ and for all $0\leq k<1/2$ the embeddings}
\Dqcgatom\cap\Deqcganom&\hookrightarrow\Hkom
\end{align*}
are compact, giving the {\MCP} for $k=0$.
\end{itemize}
\end{rem}

By \cite[Proposition 4.4, (4.21)]{goldshteinmitreairinamariushodgedecomixedbc}
and the latter remark we have:

\begin{theo}
\mylabel{MCPtheo}
Let $\om$ be a weakly Lipschitz domain and $\gat$ be an admissible patch,
i.e., let $\om$ be a (weakly) Lipschitz domain and $\gat$ be an Lipschitz patch of $\ga$.
Then the pair $\pair$ has the \MCP. 
\end{theo}

\begin{cor}
\mylabel{harmdifindim}
Let the pair $\pair$ have the \MCP. 
Then, for all $q$ the spaces 
of harmonic Dirichlet-Neumann forms $\harmdiq$ are finite dimensional.
\end{cor}

\begin{proof}
The {\MCP} implies immediately that
the unit ball in $\harmdiq$ is compact.
\end{proof}

For details about the 
particular dimensions see 
\cite{picardboundaryelectro} or 
\cite{goldshteinmitreairinamariushodgedecomixedbc}.
We note that the dimensions of $\harmdiq$
depend only on topological properties of the pair $\pair$.

\begin{lem}
\mylabel{poincarediff}
{\sf(Poincar\'e-type Estimate for Differential Forms)}
Let the pair $\pair$ have the \MCP. 
Then, for all $q$ there exist positive constants $\cpq$, 
such that
\begin{align*}
\normLtqom{E}
&\leq\cpq\big(\normLtqpoom{\ed E}^2+\normLtqmoom{\cd E}^2\big)^{1/2}
\intertext{holds for all $E\in\Dqcgatom\cap\Deqcganom\cap\harmdiq^{\bot}$. Moreover,}
\normLtqom{(\id-\pi_{q})E}
&\leq\cpq\big(\normLtqpoom{\ed E}^2+\normLtqmoom{\cd E}^2\big)^{1/2}
\end{align*}
holds for all $E\in\Dqcgatom\cap\Deqcganom$,
where $\pi_{q}:\Ltqom\to\harmdiq$ denotes the $\Ltqom$-orthogonal projection
onto the Dirichlet-Neumann forms $\harmdiq$.
\end{lem}

Here and throughout the paper, $\bot$ denotes orthogonality in $\Ltqom$.\\

\begin{proof}
A standard indirect argument utilizing the {\MCP}
yields the desired estimates.
\end{proof}

By Stokes' theorem and approximation always 
$$\Dqczgatom\subset(\cd\Deqpocganom)^{\bot},\quad
\Deqczganom\subset(\ed\Dqmocgatom)^{\bot}$$ 
hold. Equality in the latter relations is not clear
and needs another assumption on the pair $\pair$.

\begin{defi}
\mylabel{defmaxapprox}
The pair $\pair$ has the `Maxwell approximation property' {\sf(MAP)},
if for all $q$ 
$$\Dqczgatom=(\cd\Deqpocganom)^{\bot},\quad
\Deqczganom=(\ed\Dqmocgatom)^{\bot}.$$
\end{defi}

\begin{rem}
\mylabel{remmaxapproxdual}
By $*$-duality the pair $\pair$ has the \MAP, 
if and only if the pair $\pairgan$ has the \MAP, i.e.,
if and only if for all $q$ 
$$\Dqczganom=(\cd\Deqpocgatom)^{\bot},\quad
\Deqczgatom=(\ed\Dqmocganom)^{\bot}.$$
\end{rem}

\begin{rem}
\mylabel{remmaxapprox}
If $\gat=\ga$ or $\gat=\emptyset$,
the {\MAP} is simply given by the projection theorem in Hilbert spaces
and by the definitions of the respective Sobolev spaces.
For the general case $\emptyset\subset\gat\subset\ga$
with possibly $\emptyset\subsetneq\gat\subsetneq\ga$,
Jochmann proved the {\MAP} in \cite{jochmanncompembmaxmixbc}
considering the special case of a bounded domain $\om\subset\rt$.
As in Remark \ref{remmaxcomp} he needs $\om$ to be Lipschitz 
and $\gamma$ to be a Lipschitz interface.
Kuhn showed the {\MAP} in \cite{kuhndiss} for smooth domains $\om\subset\rN$ 
or even for smooth Riemannian manifolds $\om$
with smooth boundary and admissible interface $\gamma$. 
Again, a sufficient result for us has been given recently 
by Gol'dshtein and Mitrea (I. \& M.) in
\cite[Theorem 4.3, (4.16)]{goldshteinmitreairinamariushodgedecomixedbc}.
Like in Remark \ref{remmaxcomp}, for this
$\om$ has to be a weakly Lipschitz domain
and $\gat\subset\ga$ to be an admissible patch.
\end{rem}

By \cite[Theorem 4.3, (4.16)]{goldshteinmitreairinamariushodgedecomixedbc}
and the latter remark we have:

\begin{theo}
\mylabel{MAPtheo}
Let $\om$ be a weakly Lipschitz domain and $\gat$ be an admissible patch,
i.e., let $\om$ be a (weakly) Lipschitz domain and $\gat$ be an Lipschitz patch of $\ga$.
Then the pair $\pair$ has the \MAP. 
\end{theo}

\begin{lem}
\mylabel{hodgehelmdeco}
{\sf(Hodge-Helmholtz Decomposition for Differential Forms)}
Let the pair $\pair$ have the \MAP. 
Then, the orthogonal decompositions 
\begin{align*}
\Ltqom
&=\ol{\ed\Dqmocgatom}\oplus\Deqczganom\\
&=\Dqczgatom\oplus\ol{\cd\Deqpocganom}\\
&=\ol{\ed\Dqmocgatom}\oplus\harmdiq\oplus\ol{\cd\Deqpocganom}
\intertext{hold. If the pair $\pair$ has additionally the \MCP, then}
\ed\Dqmocgatom
&=\ed\big(\Dqmocgatom\cap\cd\Deqcganom\big)=\Dqczgatom\cap\harmdiq^{\bot},\\
\cd\Deqpocganom
&=\cd\big(\Deqpocganom\cap\ed\Dqcgatom\big)=\Deqczganom\cap\harmdiq^{\bot}
\intertext{and these are closed subspaces of $\Ltqom$.
Moreover, then the orthogonal decompositions}
\Ltqom
&=\ed\Dqmocgatom\oplus\Deqczganom\\
&=\Dqczgatom\oplus\cd\Deqpocganom\\
&=\ed\Dqmocgatom\oplus\harmdiq\oplus\cd\Deqpocganom
\end{align*}
hold.
\end{lem}

Here, $\oplus$ denotes the $\Ltqom$-orthogonal sum
and all closures are taken in $\Ltqom$.\\

\begin{proof}
By the projection theorem in Hilbert space
and the \MAP\, we obtain immediately 
the two $\Ltqom$-orthogonal decompositions
$$\ol{\ed\Dqmocgatom}\oplus\Deqczganom
=\Ltqom
=\Dqczgatom\oplus\ol{\cd\Deqpocganom},$$
where the closures are taken in $\Ltqom$. Since 
$$\ed\Dqmocgatom\subset\Dqczgatom,\quad\cd\Deqpocganom\subset\Deqczganom$$ 
and applying the latter decompositions separately to 
$\Deqczganom$ or $\Dqczgatom$
we get a refined decomposition, namely
\begin{align*}
\Ltqom&=\ol{\ed\Dqmocgatom}\oplus\harmdiq\oplus\ol{\cd\Deqpocganom}.
\intertext{Applying this decomposition to 
$\Dqmocgatom$ and $\Deqpocganom$ yields also}
\ed\Dqmocgatom&=\ed\big(\Dqmocgatom\cap\ol{\cd\Deqcganom}\big),\\
\cd\Deqpocganom&=\cd\big(\Deqpocganom\cap\ol{\ed\Dqcgatom}\big).
\end{align*}
Now, Lemma \ref{poincarediff}
shows that $\ed\Dqmocgatom$ and $\cd\Deqpocganom$ 
are even closed subspaces of $\Ltqom$.
Hence, we obtain the asserted 
Hodge-Helmholtz decompositions of $\Ltqom$.
\end{proof}

\subsection{Functions and Vector Fields}

We turn to the special case $q=1$, the case of vector fields,
and use the notations and identifications from 
\cite{neffpaulywitschgenkornrn}
and \cite{neffpaulywitschgenkornrt,neffpaulywitschgenkornrtzap,
neffpaulywitschgenkornrtsli}.
Especially, $\Ltqom$ can be identified
with the usual Lebesgue spaces 
of square integrable functions or vector fields on $\om$ 
with values in $\rz^{n}$, $n:=n_{N,q}:=\binom{N}{q}$, 
and will be denoted by $\Ltom:=\Lt(\om,\rz^{n})$.
We have the standard Sobolev spaces
\begin{align*}
\Hgom&:=\set{u\in\Lt(\om,\rz)}{\grad u\in\Lt(\om,\rz^{N})},\\
\Hdom&:=\set{v\in\Lt(\om,\rz^{N})}{\div v\in\Lt(\om,\rz)},\\
\Hcom&:=\set{v\in\Lt(\om,\rz^{N})}{\curl v\in\Lt(\om,\rz^{N(N-1)/2})}
\end{align*}
and by natural isomorphic identification
$$\qD{0}(\om)\cong\Hgom,\quad
\qDe{1}(\om)\cong\Hdom,\quad
\qD{1}(\om)\cong\Hcom.$$
Generally $\qD{q}(\om)\cong\qDe{N-q}(\om)$ 
holds by Hodge star duality.
For $v\in\Ciom$ and $N=3,4$
$$\curl v
=\begin{bmatrix}
\p_{2}v_{3}-\p_{3}v_{2}\\
\p_{3}v_{1}-\p_{1}v_{3}\\
\p_{1}v_{2}-\p_{2}v_{1}
\end{bmatrix}\in\rt,\quad
\curl v
=\begin{bmatrix}
\p_{1}v_{2}-\p_{2}v_{1}\\
\p_{1}v_{3}-\p_{3}v_{1}\\
\p_{1}v_{4}-\p_{4}v_{1}\\
\p_{2}v_{3}-\p_{3}v_{2}\\
\p_{2}v_{4}-\p_{4}v_{2}\\
\p_{3}v_{4}-\p_{4}v_{3}
\end{bmatrix}\in\rz^6$$
hold, whereas $\curl v=\p_{1}v_{2}-\p_{2}v_{1}\in\rz$
or $\curl v\in\rz^{10}$ for $N=2$ or $N=5$, respectively.

\begin{figure}
\begin{center}
\begin{tabular}{|c||c|c|c|c|}
\hline
$q$ & $0$ & $1$ & $2$ & $3$\\
\hline\hline
$\ed$ & $\grad$ & $\curl$ & $\div$ & $0$\\
\hline
$\cd$ & $0$ & $\div$ & $-\curl$ & $\grad$\\
\hline
$\Dqcgatom$ & $\Hgcgatom$ & $\Hccgatom$ & $\Hdcgatom$ & $\Ltom$\\
\hline
$\Deqcganom$ & $\Ltom$ & $\Hdcganom$ & $\Hccganom$ & $\Hgcganom$\\
\hline
$\iota_{\gat}^{*}E$ & $E|_{\gat}$ & $\nu\times E|_{\gat}$ & $\nu\cdot E|_{\gat}$ & $0$\\
\hline
$\circledast\iota_{\gan}^{*}*E$ & $0$ & $\nu\cdot E|_{\gan}$ & $-\nu\times(\nu\times E)|_{\gan}$ & $E|_{\gan}$\\[0.5mm]
\hline
\end{tabular}
\end{center}
\caption{identification table for $q$-forms and vector proxies in $\rt$}
\mylabel{qformtable}
\end{figure}

Moreover, we have the closed subspaces 
$$\Hgcgatom,\quad\Hccgatom,\quad\Hdcganom,$$ 
in which the homogeneous scalar, tangential and normal boundary conditions
$$u|_{\gat}=0,\quad\nu\times v|_{\gat}=0,\quad\nu\cdot v|_{\gan}=0$$
are generalized, as reincarnations of 
$\qDc{0}(\gat,\om)$, $\qDc{1}(\gat,\om)$ and $\qDec{1}(\gan,\om)$,
respectively. Here $\nu$ denotes the outer unit normal at $\ga$.
If $\gat=\ga$ (and $\gan=\emptyset$)
resp. $\gat=\emptyset$ (and $\gan=\ga$) 
we obtain the usual Sobolev spaces
$$\Hgcom,\quad\Hccom,\quad\Hdom$$
resp. 
$$\Hgom,\quad\Hcom,\quad\Hdcom.$$
We note that $\Hgom$ and $\Hgcom$
coincide with the usual standard Sobolev spaces 
$\Hoom$ and $\Hocom$, respectively.
As before, the index $0$, now attached to the symbols $\curl$ or $\div$,
indicates vanishing $\curl$ or $\div$, e.g.,
\begin{align*}
\Hcczgatom&=\set{v\in\Hccgatom}{\curl v=0},\\
\Hdzom&=\set{v\in\Hdom}{\div v=0}.
\end{align*}
Finally, we denote the `harmonic Dirichlet-Neumann fields' by
$$\harmdio\cong\harmdi:=\Hcczgatom\cap\Hdczganom.$$

Assuming the {\MCP} for the pair $\pair$,
then $\harmdi$ is finite dimensional by Corollary \ref{harmdifindim}
and we have the two (out of four) compact embeddings
\begin{align}
\Hgcgatom&\hookrightarrow\Ltom,\mylabel{rellich}\\
\Hccgatom\cap\Hdcganom&\hookrightarrow\Ltom,\mylabel{maxc}
\end{align}
i.e., Rellich's selection theorem ($q=0$) and
the vectorial Maxwell's compactness property ($q=1$).
Moreover, by Lemma \ref{poincarediff}
we get the following Poincar\'e and Maxwell estimates:

\begin{cor}
\mylabel{poincare}
{\sf(Poincar\'e Estimate for Functions)}
Let the pair $\pair$ have the {\MCP} and $\cp:=\cpz$. Then
\begin{align*}
\normLtom{u}&\leq\cp\normLtom{\grad u}
\intertext{holds for all $u\in\Hgcgatom$ if $\gat\neq\emptyset$
and for all $u\in\Hgom\cap\rz^{\bot}$ if $\gat=\emptyset$.
Moreover, for all $u\in\Hgom$}
\normLtom{(\id-\pi_{0})u}&\leq\cp\normLtom{\grad u}
\end{align*}
holds, where $\pi_{0}:\Ltom\to\rz$ denotes the 
$\Ltom$-orthogonal projection onto the constants.
\end{cor}

We note that if $\gat\neq\emptyset$ we have $\harmdiz=\{0\}$.
Furthermore, $\harmdiz=\rz$ and $\Hgcgatom=\Hgom$ hold if $\gat=\emptyset$.

\begin{cor}
\mylabel{poincaremax}
{\sf(Maxwell Estimate for Vector Fields)}
Let the pair $\pair$ have the {\MCP} and $\cm:=\cpo$. Then
\begin{align*}
\normLtom{v}
&\leq\cm\big(\normLtom{\curl v}^2+\normLtom{\div v}^2\big)^{1/2}
\intertext{holds for all $v\in\Hccgatom\cap\Hdcganom\cap\harmdi^{\bot}$ as well as}
\normLtom{(\id-\pi_{1})v}
&\leq\cm\big(\normLtom{\curl v}^2+\normLtom{\div v}^2\big)^{1/2}
\end{align*}
holds for all $v\in\Hccgatom\cap\Hdcganom$,
where again $\pi_{1}:\Ltom\to\harmdi$ denotes the 
$\Ltom$-orthogonal projection onto the Dirichlet-Neumann fields $\harmdi$.
\end{cor}

Lemma \ref{hodgehelmdeco} yields:

\begin{cor}
\mylabel{helmdeco}
{\sf(Helmholtz Decompositions for Vector Fields)}
Let the pair $\pair$ have the {\MCP} and the \MAP.
Then, the orthogonal decompositions
\begin{align*}
\Ltom
&=\grad\Hgcgatom\oplus\Hdczganom\\
&=\Hcczgatom\oplus\big(\Hdczganom\cap\harmdi^{\bot}\big)
\end{align*} 
hold.
\end{cor}

\subsection{Tensor Fields}

Next, we extend our calculus to tensor fields, i.e., matrix fields.
For vector fields $v$ with components in $\Hgom$
and tensor fields $\T$ with rows in $\Hcom$ resp. $\Hdom$, i.e.,
$$v=\dvec{v_{1}}{\vdots}{v_{N}},\quad v_{n}\in\Hgom,\quad
\T=\dvec{\trans{\T}_{1}}{\vdots}{\trans{\T}_{N}},\quad\T_{n}\in\Hcom\text{ resp. }\Hdom$$
for $n=1,\dots,N$ we define (in the weak sense)
$$\Grad v:=\dvec{\trans{\grad}v_{1}}{\vdots}{\trans{\grad}v_{N}}=J_{v},\quad
\Curl\T:=\dvec{\trans{\curl}\T_{1}}{\vdots}{\trans{\curl}\T_{N}},\quad
\Div\T:=\dvec{\div\T_{1}}{\vdots}{\div\T_{N}},$$ 
where $J_{v}$\footnote{Sometimes, the Jacobian
$J_{v}$ is also denoted by $\na v$.} 
denotes the Jacobian of $v$ 
and $\trans{}$ the transpose.
We note that $v$ and $\Div\T$ are $N$-vector fields,
$\T$ and $\Grad v$ are $(N\times N)$-tensor fields, 
whereas $\Curl\T$ is a $(N\times N(N-1)/2)$-tensor field.
The corresponding Sobolev spaces will be denoted by
$$\HGom,\quad\HCom,\quad\HCzom,\quad\HDom,\quad\HDzom$$
and
$$\HGcgatom,\quad\HCcgatom,\quad\HCczgatom,\quad\HDcganom,\quad\HDczganom,$$
again with the usual notations if $\gat\in\{\emptyset,\ga\}$.

From Corollaries \ref{poincare},
\ref{poincaremax} and \ref{helmdeco} we obtain immediately:

\begin{cor}
\mylabel{poincarevec}
{\sf(Poincar\'e Estimate for Vector Fields)}
Let the pair $\pair$ have the \MCP. Then
\begin{align*}
\normLtom{v}&\leq\cp\normLtom{\Grad v}
\intertext{holds for all $v\in\HGcgatom$ 
if $\gat\neq\emptyset$
and for all $v\in\HGom\cap(\rN)^{\bot}$ if $\gat=\emptyset$.
Moreover, for all $v\in\HGom$}
\normLtom{(\id-\pi_{0}^N)v}&\leq\cp\normLtom{\Grad v}
\end{align*}
holds, where $\pi_{0}^N:\Ltom\to\rN$ denotes the 
$\Ltom$-orthogonal projection onto $\rN$.
\end{cor}

\begin{cor}
\mylabel{poincaremaxten}
{\sf(Maxwell Estimate for Tensor Fields)}
Let the pair $\pair$ have the \MCP. Then
\begin{align*}
\normLtom{\T}
&\leq\cm\big(\normLtom{\Curl\T}^2+\normLtom{\Div\T}^2\big)^{1/2}
\intertext{holds for all $\T\in\HCcgatom\cap\HDcganom\cap(\harmdi^N)^{\bot}$ as well as}
\normLtom{(\id-\pi_{1}^{N})\T}
&\leq\cm\big(\normLtom{\Curl\T}^2+\normLtom{\Div\T}^2\big)^{1/2}
\end{align*}
holds for all $\T\in\HCcgatom\cap\HDcganom$,
where $\pi_{1}^{N}:\Ltom\to\harmdi^{N}$ denotes the 
$\Ltom$-orthogonal projection onto the ($N$-times)-Dirichlet-Neumann fields $\harmdi^N$.
\end{cor}

\begin{cor}
\mylabel{helmdecoten}
{\sf(Helmholtz Decompositions for Tensor Fields)}
Let the pair $\pair$ have the {\MCP} and the \MAP.
Then, the orthogonal decompositions
\begin{align*}
\Ltom
&=\Grad\HGcgatom\oplus\HDczganom\\
&=\HCczgatom\oplus\big(\HDczganom\cap(\harmdi^{N})^{\bot}\big)
\end{align*} 
hold.
\end{cor}

We also need Korn's First Inequality.

\begin{defi}
\mylabel{secondkorn}
{\sf(Korn's Second Inequality)}
The domain $\om$ has the `Korn property' {\sf(KP)}, if
\begin{itemize}
\item[\bf(i)]
Korn's second inequality holds, this is,
there exists a constant $c>0$, such that for all vector fields $v\in\HGom$
$$c\normLtom{\Grad v}\leq\normLtom{v}+\normLtom{\sym\Grad v},$$
\item[\bf(ii)]
and Rellich's selection theorem holds for $\Hgom$, this is,
the natural embedding $\Hgom\hookrightarrow\Ltom$ is compact.
\end{itemize}
\end{defi}

Here, we introduce the symmetric and skew-symmetric parts 
$$\sym\T:=\frac{1}{2}(\T+\trans{\T}),\quad
\skew\T:=\T-\sym\T=\frac{1}{2}(\T-\trans{\T})$$
of a tensor field $\T=\skew\T+\sym\T$\footnote{Note that $\sym\T$ and $\skew\T$
are point-wise orthogonal with respect to the standard inner product in $\rNtN$.}. 

\begin{rem}
\mylabel{secondkornrem}
There exists a rich amount of literature for the \KP,
which we do not intend to list here.
We refer to our overview on Korn's inequalities in \cite{neffpaulywitschgenkornrtsli}.
\end{rem}

\begin{theo}
\mylabel{secondkorntheo}
Korn's second inequality holds for domains $\om$ having the strict cone property.
For domains $\om$ with the segment property, Rellich's selection theorem for $\Hgom$ is valid.
Thus, e.g., Lipschitz domains $\om$ possess the \KP.
\end{theo}

\begin{proof}
Book of Leis \cite{leisbook}.
\end{proof}

By a standard indirect argument we immediately obtain:

\begin{cor}
\mylabel{korn}
{\sf(Korn's First Inequality: Standard Version)}
Let $\om$ have the \KP. Then, there exists a constant $\cks>0$ 
such that the following holds:
\begin{itemize}
\item[\bf(i)] If $\gat\neq\emptyset$ then
\begin{align}
\mylabel{firstkornstandard}
(1+\cp^2)^{-1/2}\norm{v}_{\Hoom}\leq\normLtom{\Grad v}\leq\cks\normLtom{\sym\Grad v}
\end{align}
holds for all vector fields $v\in\HGcgatom$.
\item[\bf(ii)] If $\gat=\emptyset$, 
then the inequalities \eqref{firstkornstandard}
hold for all vector fields $v\in\HGom$ with $\Grad v\bot\soN$ and $v\bot\rN$.
Moreover, the second inequality of \eqref{firstkornstandard} holds
for all vector fields $v\in\HGom$ with $\Grad v\bot\soN$.
For all $v\in\HGom$
\begin{align}
\mylabel{firstkornstandardempty}
(1+\cp^2)^{-1/2}\norm{v-r_{v}}_{\Hoom}\leq\normLtom{\Grad v-\Tskew_{\Grad v}}
\leq\cks\normLtom{\sym\Grad v}
\end{align}
holds, where the ridgid motion $r_{v}$ and the skew-symmetric tensor 
$\Tskew_{\Grad v}=\Grad r_{v}$ are given by $r_{v}(x):=\Tskew_{\Grad v}x+b_{v}$ and
$$\Tskew_{\Grad v}:=\skew\oint_{\om}\Grad v\dl\in\soN,\quad
b_{v}:=\oint_{\om}v\dl-\Tskew_{\Grad v}\oint_{\om}x\dl_{x}\in\rN.$$
We note $v-r_{v}\bot\,\rN$ and $\Grad(v-r_{v})=\Grad v-\Tskew_{\Grad v}\bot\soN$.
\end{itemize}
\end{cor}

Here, we generally define
$$\oint_{\om}u\dl:=\lambda(\om)^{-1}\int_{\om}u\dl,\quad\lambda\text{ Lebesgue's measure.}$$
We note that $\Tskew_{\Grad v}=\pi_{\soN}\Grad v\in\soN$, where
$\pi_{\soN}:\Ltom\to\soN$ denotes the $\Ltom$-orthogonal projection 
onto the constant skew-symmetric tensor fields $\soN$. 
Moreover, we have generally 
for square integrable $(N\times N)$-tensor fields $\T$
\begin{align}
\mylabel{defTpi}
\pi_{\soN}\T:=\Tskew_{\T}:=\skew\oint_{\om}\T\dl\in\soN.
\end{align}

\subsection{Sliceable and Admissible Domains}

The essential tools to prove 
our main result Theorem \ref{maintheo} are 
\begin{itemize}
\item the Maxwell estimate for tensor fields (Corollary \ref{poincaremaxten}),
\item the Helmholtz decomposition for tensor fields (Corollary \ref{helmdecoten}),
\item and a generalized version of Korn's first inequality (Corollary \ref{korn}).
\end{itemize}
For the first two tools the pair $\pair$ needs to have the {\MCP} and the \MAP.
The third tool will be provided in Lemma \ref{genkornlem}
and needs at least the \KP.
As already pointed out, these three properties hold, e.g., 
for Lipschitz domains $\om$
and admissible boundary patches $\gat$.
Moreover, we will make use of the fact that any irrotational vector field
is already a gradient if the underlying domain is simply connected.
For this, we present a trick, the concept of sliceable domains, 
which we have used already in \cite{neffpaulywitschgenkornrtsli}.

\begin{figure}
\center
\begin{minipage}{7cm}
\center
\includegraphics[scale=0.08]{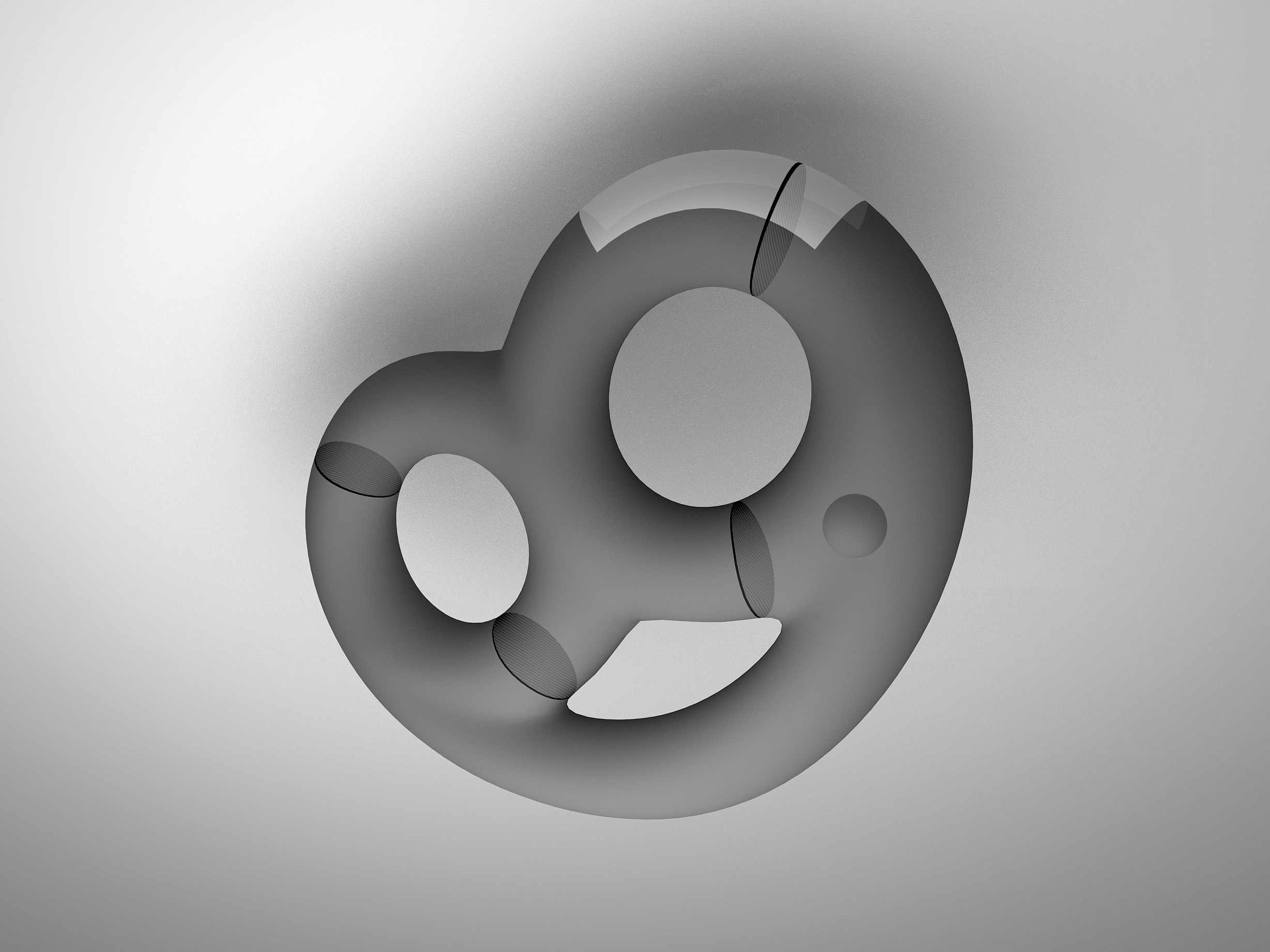}
\vspace*{-8mm}
\begin{tikzpicture}
\fill[line width=3pt,gray,rounded corners=20pt,shading=radial]
(0.5,1.5)--(1,2)--(3,1.5)--(5,2)--(6,1)--(5,0)--(3,0.5)--(1,0)--(0,1)--(1,2)--(0.5,1.5);
\draw[line width=3pt,rounded corners=20pt]
(0.5,1.5)--(1,2)--(3,1.5)--(5,2)--(6,1)--(5,0)--(3,0.5)--(1,0)--(0,1)--(1,2)--(0.5,1.5);
\draw[line width=3pt,lightgray,rounded corners=20pt](0.5,1.5)--(1,2)--(2,1.752);
\draw[line width=1pt](1.2,1.89)--(1.3,1.0)--(0.9,0.2);
\draw[line width=1pt](1.8,1.855)..controls(1.5,0.5)and(3,0.5)..(5,1.8);
\draw[line width=1pt,dashed](1.6,1.90)..controls(1.6,-0.7)and(4,1.8)..(4.32,0.2);
\fill[white](1.3,1)circle(0.3);
\draw[line width=3pt](1.3,1)circle(0.3);
\fill[white](4.6,1.4)circle(0.3);
\draw[line width=3pt](4.6,1.4)circle(0.3);
\end{tikzpicture}
\vspace*{-8mm}
\begin{tikzpicture}
\fill[line width=3pt,gray,rounded corners=20pt,shading=radial]
(0.5,1.5)--(1,2)--(3,1.5)--(5,2)--(6,1)--(5,0)--(3,0.5)--(1,0)--(0,1)--(1,2)--(0.5,1.5);
\draw[line width=3pt,rounded corners=20pt]
(0.5,1.5)--(1,2)--(3,1.5)--(5,2)--(6,1)--(5,0)--(3,0.5)--(1,0)--(0,1)--(1,2)--(0.5,1.5);
\draw[line width=3pt,lightgray,rounded corners=20pt](0.5,1.5)--(1,2)--(2,1.752);
\draw[line width=1pt](1.2,1.89)--(1.3,1.0)--(0.9,0.2);
\draw[line width=1pt](1.8,1.855)..controls(1.5,0.5)and(3,0.5)..(5,1.8);
\draw[line width=1pt](1.6,1.90)..controls(1.6,-0.7)and(4,1.8)..(4.32,0.2);
\fill[white](1.3,1)circle(0.3);
\draw[line width=3pt](1.3,1)circle(0.3);
\fill[white](4.6,1.4)circle(0.3);
\draw[line width=3pt](4.6,1.4)circle(0.3);
\fill[white](4.1,0.7)circle(0.3);
\draw[line width=3pt](4.1,0.7)circle(0.3);
\end{tikzpicture}
\begin{tikzpicture}
\fill[line width=3pt,gray,rounded corners=20pt,shading=radial]
(0.5,1.5)--(1,2)--(3,1.5)--(5,2)--(6,1)--(5,0)--(3,0.5)--(1,0)--(0,1)--(1,2)--(0.5,1.5);
\draw[line width=3pt,rounded corners=20pt]
(0.5,1.5)--(1,2)--(3,1.5)--(5,2)--(6,1)--(5,0)--(3,0.5)--(1,0)--(0,1)--(1,2)--(0.5,1.5);
\draw[line width=3pt,lightgray,rounded corners=20pt](0.5,1.5)--(1,2)--(2,1.752);
\draw[line width=1pt](1.2,1.89)--(1.3,1.0)--(0.9,0.2);
\draw[line width=1pt](1.8,1.855)..controls(1.5,1.5)and(3,0.5)..(5,1.8);
\draw[line width=1pt](1.6,1.90)..controls(1.6,-0.7)and(4,1.8)..(5.5,0.45);
\fill[white](1.3,1)circle(0.3);
\draw[line width=3pt](1.3,1)circle(0.3);
\fill[white](4.6,1.4)circle(0.3);
\draw[line width=3pt](4.6,1.4)circle(0.3);
\fill[white](2.6,1.2)circle(0.2);
\draw[line width=3pt](2.6,1.2)circle(0.2);
\fill[white](3.6,1.2)circle(0.2);
\draw[line width=3pt](3.6,1.2)circle(0.2);
\fill[white](2.2,0.7)circle(0.2);
\draw[line width=3pt](2.2,0.7)circle(0.2);
\fill[white](3.2,0.75)circle(0.2);
\draw[line width=3pt](3.2,0.75)circle(0.2);
\fill[white](4.2,0.8)circle(0.2);
\draw[line width=3pt](4.2,0.8)circle(0.2);
\fill[white](5.2,0.8)circle(0.2);
\draw[line width=3pt](5.2,0.8)circle(0.2);
\end{tikzpicture}
\end{minipage}
\begin{minipage}{7cm}
\vspace*{-6mm}
\center
\includegraphics[scale=0.08]{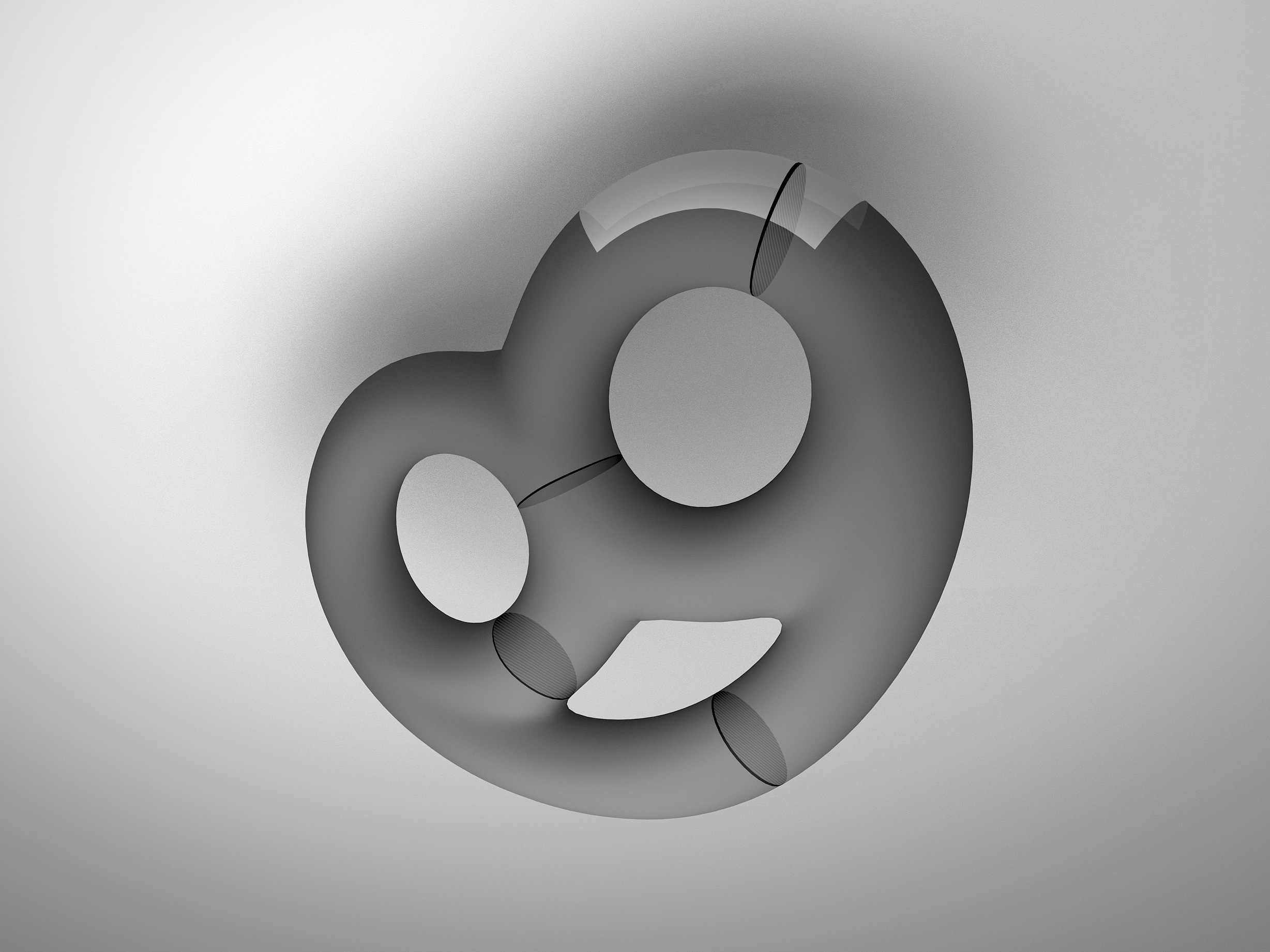}
\vspace*{-1mm}
\begin{tikzpicture}
\fill[line width=3pt,gray,rounded corners=20pt,shading=axis]
(0.5,1.5)--(1,2)--(3,1.5)--(5,2)--(6,1)--(5,0)--(3,0.5)--(1,0)--(0,1)--(1,2)--(0.5,1.5);
\draw[line width=3pt,rounded corners=20pt]
(0.5,1.5)--(1,2)--(3,1.5)--(5,2)--(6,1)--(5,0)--(3,0.5)--(1,0)--(0,1)--(1,2)--(0.5,1.5);
\draw[line width=3pt,lightgray,rounded corners=20pt](0.5,1.5)--(1,2)--(2,1.752);
\draw[line width=1pt](1.2,1.89)--(1.3,1.0)--(4.6,1)--(4.3,0.15);
\fill[white](1.3,1)circle(0.3);
\draw[line width=3pt](1.3,1)circle(0.3);
\fill[white](4.6,1)circle(0.3);
\draw[line width=3pt](4.6,1)circle(0.3);
\end{tikzpicture}
\vspace*{-1mm}
\begin{tikzpicture}
\fill[line width=3pt,gray,rounded corners=20pt,shading=axis]
(0.5,1.5)--(1,2)--(3,1.5)--(5,2)--(6,1)--(5,0)--(3,0.5)--(1,0)--(0,1)--(1,2)--(0.5,1.5);
\draw[line width=3pt,rounded corners=20pt]
(0.5,1.5)--(1,2)--(3,1.5)--(5,2)--(6,1)--(5,0)--(3,0.5)--(1,0)--(0,1)--(1,2)--(0.5,1.5);
\draw[line width=3pt,lightgray,rounded corners=20pt](0.5,1.5)--(1,2)--(2,1.752);
\draw[line width=1pt,rounded corners=0pt]
(1.2,1.89)--(1.3,1.0)--(4.1,0.6)--(4.6,1.5)--(5.5,0.45);
\fill[white](1.3,1)circle(0.3);
\draw[line width=3pt](1.3,1)circle(0.3);
\fill[white](4.6,1.4)circle(0.3);
\draw[line width=3pt](4.6,1.4)circle(0.3);
\fill[white](4.1,0.7)circle(0.3);
\draw[line width=3pt](4.1,0.7)circle(0.3);
\end{tikzpicture}
\begin{tikzpicture}
\fill[line width=3pt,gray,rounded corners=20pt,shading=axis]
(0.5,1.5)--(1,2)--(3,1.5)--(5,2)--(6,1)--(5,0)--(3,0.5)--(1,0)--(0,1)--(1,2)--(0.5,1.5);
\draw[line width=3pt,rounded corners=20pt]
(0.5,1.5)--(1,2)--(3,1.5)--(5,2)--(6,1)--(5,0)--(3,0.5)--(1,0)--(0,1)--(1,2)--(0.5,1.5);
\draw[line width=3pt,lightgray,rounded corners=20pt](0.5,1.5)--(1,2)--(2,1.752);
\draw[line width=1pt,rounded corners=0pt]
(1.2,1.89)--(1.3,1.0)--(2.2,0.7)--(2.6,1.3)--(3.2,0.7)--(3.6,1.2)--(4.2,0.8)--(4.6,1.5)--(5.5,0.45);
\fill[white](1.3,1)circle(0.3);
\draw[line width=3pt](1.3,1)circle(0.3);
\fill[white](4.6,1.4)circle(0.3);
\draw[line width=3pt](4.6,1.4)circle(0.3);
\fill[white](2.6,1.2)circle(0.2);
\draw[line width=3pt](2.6,1.2)circle(0.2);
\fill[white](3.6,1.2)circle(0.2);
\draw[line width=3pt](3.6,1.2)circle(0.2);
\fill[white](2.2,0.7)circle(0.2);
\draw[line width=3pt](2.2,0.7)circle(0.2);
\fill[white](3.2,0.75)circle(0.2);
\draw[line width=3pt](3.2,0.75)circle(0.2);
\fill[white](4.2,0.8)circle(0.2);
\draw[line width=3pt](4.2,0.8)circle(0.2);
\fill[white](5.2,0.8)circle(0.2);
\draw[line width=3pt](5.2,0.8)circle(0.2);
\end{tikzpicture}
\end{minipage}
\vspace*{-8mm}
\caption{Some ways to `cut' sliceable domains $\om$ in $\rt$ and $\rz^2$
into two ($J=2$) or more ($J=3,4$) `pieces'.
The boundary part $\gat$ is colored in light gray.
Roughly speaking, a domain is sliceable
if it can be cut into finitely many simply connected Lipschitz pieces $\om_{j}$, i.e., 
any closed curve inside some piece $\om_{j}$
is homotop to a point, this is, one has to cut all `handles'. 
In three and higher dimensions, holes inside $\om$ are permitted,
but this is forbidden in the two-dimensional case.
Note that, in these examples it is always possible to slice $\om$
into two ($J=2$) pieces.}
\mylabel{brezel}
\end{figure}

\begin{defi}
\mylabel{domdefsli}
The pair $\pair$ is called `sliceable', 
if there exist $J\in\nz$ and $\om_{j}\subset\om$, $j=1,\dots,J$, 
such that $\om\setminus(\om_{1}\cup\ldots\cup\om_{J})$ has zero Lebesgue-measure
and for $j=1,\dots,J$
\begin{itemize}
\item[\bf(i)] 
$\om_{j}$ are open, disjoint and simply connected subdomains of $\om$ having the \KP,
\item[\bf(ii)] 
$\gatj:=\interior_{\mathsf{rel}}(\ol{\om_{j}}\cap\gat)\neq\emptyset$, if $\gat\neq\emptyset$.
\end{itemize}
\end{defi}

Here, $\interior_{\mathsf{rel}}$ denotes the interior with respect to
the topology on $\ga$.

\begin{rem}
\mylabel{slicerem}
From a practical point of view, 
\ul{all} domains considered in applications are sliceable,
but it is unclear whether every Lipschitz pair $\pair$ is already sliceable.
\end{rem}

Now, we can introduce our general assumptions 
on the domain and its boundary parts.

\begin{defi}
\mylabel{domdef}
The pair $\pair$ is called `admissible', if
\begin{itemize}
\item the pair $\pair$ possesses the {\MCP}and the \MAP,
\item and the pair $\pair$ is sliceable.
\end{itemize}
\end{defi}

\begin{rem}
\mylabel{domdefrem}
In particular, the pair $\pair$ is admissible if
\begin{itemize}
\item $\om$ has a Lipschitz boundary $\ga$,
\item $\gat$ is a Lipschitz patch,
\item $\pair$ is sliceable.
\end{itemize}
\end{rem}

\section{Proofs}

Let the pair $\pair$ be \ul{admissible}.
On our way to prove our main result
we follow in close lines the arguments of
\cite[section 3]{neffpaulywitschgenkornrtsli}.
First we prove a non-standard version of Korn's first inequality Corollary \ref{korn},
which will be presented as Lemma \ref{genkornlem}.
Then, we prove our main result.
Although, all subsequent proofs are very similar to the ones 
given in \cite[Lemmas 8, 9, 12, Theorem 14]{neffpaulywitschgenkornrtsli},
we will repeat them here for the convenience of the reader.

\begin{lem}
\mylabel{constlem}
Let $\gat\neq\emptyset$ and $u\in\Hgom$ with $\grad u\in\Hcczgatom$.
Then, $u$ is constant on any connected component of $\gat$.
\end{lem}

\begin{proof}
Let $x\in\gat$ and $B_{2r}:=B_{2r}(x)$ be the open ball of radius $2r>0$
around $x$ such that $B_{2r}$ is covered by a Lipschitz-chart domain 
and $\ga\cap B_{2r}\subset\gat$.
Moreover, we pick a cut-off function $\varphi\in\Cic(B_{2r})$ with $\varphi|_{B_{r}}=1$.
Then, $\varphi\grad u\in\Hcgen{}{\circ}{\om\cap B_{2r}}$.
Thus, the extension by zero $v$ of $\varphi\grad u$ to $B_{2r}$
belongs to $\Hcgen{}{}{B_{2r}}$. 
Hence, $v|_{B_{r}}\in\Hcgen{0}{}{B_{r}}$.
Since $B_{r}$ is simply connected, 
there exists a $\tilde{u}\in\Hggen{}{}{B_{r}}$
with $\grad\tilde{u}=v$ in $B_{r}$. 
In $B_{r}\setminus\ol{\om}$ we have $v=0$.
Therefore, $\tilde{u}|_{B_{r}\setminus\ol{\om}}=\tilde{c}$ 
with some $\tilde{c}\in\rz$.
Moreover, $\grad u=v=\grad\tilde{u}$ holds in $B_{r}\cap\om$,
which yields $u=\tilde{u}+c$ in $B_{r}\cap\om$ with some $c\in\rz$.
Finally, $u|_{B_{r}\cap\gat}=\tilde{c}+c$ is constant.
Therefore, $u$ is locally constant and hence the assertion follows.
\end{proof}

\begin{lem}
\mylabel{kornlem}
{\sf(Korn's First Inequality: Tangential Version)}
Let $\gat\neq\emptyset$. Then, there exists a constant $\ckt\geq\cks$, such that
$$\normLtom{\Grad v}\leq\ckt\normLtom{\sym\Grad v}$$
holds for all $v\in\HGom$ with $\Grad v\in\HCczgatom$.
\end{lem}

In classical terms, $\Grad v\in\HCczgatom$ 
means that $\grad v_{n}=\na v_{n}$, $n=1,\dots,N$, are normal at $\gat$.\\

\begin{proof}
We pick a relatively open connected component $\tilde{\ga}\neq\emptyset$ of $\gat$.
Then, there exists a constant vector $c_{v}\in\rt$ 
such that $v-c_{v}$ belongs to $\HGgen{}{\circ}{\tilde{\ga},\om}$
by Lemma \ref{constlem} applied to each component of $v$.
Corollary \ref{korn} (i) (with $\gat=\tilde{\ga}$ and a possibly larger $\ckt$) 
completes the proof.
\end{proof}

Now, we extend Korn's first inequality 
from gradient to merely irrotational tensor fields.

\begin{lem}
\mylabel{genkornlem}
{\sf(Korn's First Inequality: Irrotational Version)}
There exists $\ck\geq\ckt>0$, 
such that the following inequalities hold:
\begin{itemize}
\item[\bf(i)] If $\gat\neq\emptyset$, then for all tensor fields $\T\in\HCczgatom$
\begin{align}
\mylabel{estTsymT}
\normLtom{\T}\leq\ck\normLtom{\sym\T}.
\end{align}
\item[\bf(ii)] If $\gat=\emptyset$, then for all tensor fields $\T\in\HCzom$
there exists a piece-wise constant skew-symmetric tensor field $\Tskew$ such that
$$\normLtom{\T-\Tskew}\leq\ck\normLtom{\sym\T}.$$
\item[\bf(ii')] 
If $\gat=\emptyset$ and $\om$ is additionally simply connected, 
then (ii) holds with the uniquely determined 
constant skew-symmetric tensor field $\Tskew:=\Tskew_{\T}=\pi_{\soN}\T$
given by \eqref{defTpi}.
Moreover, $\T-\Tskew_{\T}\in\HCzom\cap\soN^{\bot}$ and
$\Tskew_{\T}=0$ if and only if $\T\bot\soN$. 
Thus, \eqref{estTsymT} holds for all $\T\in\HCzom\cap\soN^{\bot}$ as well.
\end{itemize}
\end{lem}

Again we note that in classical terms 
a tensor $\T\in\HCczgatom$ is irrotational 
and the vector field $\T\tau|_{\gat}$ vanishes 
for all tangential vector fields $\tau$ at $\ga$.
Moreover, the sliceability of $\pair$ is precisely
needed for Lemma \ref{genkornlem} to hold.\\ 

\begin{proof}
We start with proving (i).
Let $\gat\neq\emptyset$ and $\T\in\HCczgatom$. 
We choose a sequence $(\T^{\ell})\subset\Cic(\gat;\om)$ 
converging to $\T$ in $\HCom$.
According to Definition \ref{domdefsli} 
we decompose $\om$ into $\om_{1}\cup\ldots\cup\om_{J}$
and pick some $1\leq j\leq J$.
Then, the restriction $\T_{j}:=\T|_{\om_{j}}$ 
belongs to $\HCgen{0}{}{\om_{j}}$
and $(\T^{\ell}|_{\ol{\om_{j}}})\subset\Cic(\gatj;\om)$ 
converges to $\T_{j}$ in $\HCgen{}{}{\om_{j}}$.
Thus, $\T_{j}\in\HCgen{0}{\circ}{\gatj,\om_{j}}$.
Since $\om_{j}$ is simply connected,
there exists a potential vector field $v_{j}\in\HGgen{}{}{\om_{j}}$
with $\Grad v_{j}=\T_{j}$
and Lemma \ref{kornlem} yields
\begin{align*}
\norm{\T_{j}}_{\Lt(\om_j)}
&\leq\cktj\norm{\sym\T_{j}}_{\Lt(\om_j)},&
\cktj&\,>0.
\intertext{This can be done for each $j$. Summing up, we obtain}
\normLtom{\T}
&\leq\ck\normLtom{\sym\T},&
\ck&:=\max_{j=1,\dots,J}\cktj,
\end{align*}
proving (i). 
Now, we assume $\gat=\emptyset$.
To show (ii),
let $\T\in\HCzom$ and, as before, let $\om$ be decomposed into
$\om_{1}\cup\ldots\cup\om_{J}$ by Definition \ref{domdefsli}.
Again, since every $\om_{j}$ is simply connected
and $\T_{j}\in\HCgen{0}{}{\om_{j}}$,
there exist vector fields $v_{j}\in\HGgen{}{}{\om_{j}}$ 
with $\Grad v_{j}=:\T_{j}=\T$ in $\om_{j}$.
By Korn's first inequality, Corollary \eqref{korn} (ii), 
there exist positive $\cksj$ and $\Tskew_{\T_{j}}\in\soN$ with
$$\norm{\T_{j}-\Tskew_{\T_{j}}}_{\Lt(\om_{j})}
\leq\cksj\norm{\sym\T_{j}}_{\Lt(\om_{j})},\quad
\Tskew_{\T_{j}}=\skew\oint_{\om_{j}}\T_{j}\dl=\skew\oint_{\om_{j}}\T\dl.$$
We define the piece-wise constant skew-symmetric tensor field $\Tskew$ a.e. by
$\Tskew|_{\om_{j}}:=\Tskew_{\T_{j}}$ and set $\ds\ck:=\max_{j=1,\ldots,J}\cksj$.
Summing up, gives (ii). 
We have also proved the first assertion of (ii'),
since we do not have to slice if $\om$ is simply connected.
The remaining assertion of (ii')
are trivial, since $\pi_{\soN}:\Ltom\to\soN$ 
is a $\Ltom$-orthogonal projector.
We note that this can be seen also by direct calculations:
To show that $\T-\Tskew_{\T}$ belongs to $\HCzom\cap\soN^{\bot}$ 
we note $\Tskew_{\T}\in\HCzom$ and compute for all $\Tskew\in\soN$
\begin{align*}
\scpLtom{\Tskew_{\T}}{\Tskew}
&=\scps{\int_{\om}\T\dl}{\Tskew}_{\rNtN}
=\int_{\om}\scp{\T}{\Tskew}_{\rNtN}\dl
=\scpLtom{\T}{\Tskew}.
\end{align*}
Hence, $\Tskew_{\T}=0$ implies $\T\bot\soN$.
On the other hand, setting $\Tskew:=\Tskew_{\T}$ shows that
$\T\bot\soN$ also implies $\Tskew_{\T}=0$.
\end{proof}

We are ready to prove our main theorem.

\begin{proofof}{Theorem \ref{maintheo}}
Let $\gat\neq\emptyset$ and $\T\in\HCcgatom$.
By Corollary \ref{helmdecoten} we have
$$\T=\TR+\TS\in\HCczgatom\oplus\big(\HDczganom\cap(\harmdi^{N})^{\bot}\big).$$
Moreover, by Corollary \ref{poincaremaxten} we obtain
\begin{align}
\mylabel{estpsi}
\normLtom{\TS}
&\,\leq\cm\normLtom{\Curl\T}
\intertext{since $\Curl\TS=\Curl\T$ and 
$\TS\in\HCcgatom\cap\HDczganom\cap(\harmdi^N)^{\bot}$.
Then, by orthogonality, Lemma \ref{genkornlem} (i) for $\TR$ and \eqref{estpsi}}
\normLtom{\T}^2=\normLtom{\TR}^2+\normLtom{\TS}^2
&\,\leq\ck^2\normLtom{\sym\TR}^2+\normLtom{\TS}^2\nonumber\\
&\,\leq2\ck^2\normLtom{\sym\T}^2+(1+2\ck^2)\normLtom{\TS}^2\nonumber\\
&\,\leq\cone^2\big(\normLtom{\sym\T}^2+\normLtom{\Curl\T}^2\big)\nonumber
\intertext{with}
\mylabel{explicitconst}
\cone&:=\max\{\sqrt{2}\ck,\cm\sqrt{1+2\ck^2}\}
\end{align}
follows, which proves (i).
Now, let $\gat=\emptyset$ and $\T\in\HCom$.
First, we show (ii').
We follow in close lines the first part of the proof.
For the convenience of the reader, we repeat the previous 
arguments in this special case. 
According to Corollary \ref{helmdecoten} we orthogonally decompose 
$$\T=\TR+\TS\in\HCzom\oplus\big(\HDczom\cap(\harmdi^N)^{\bot}\big).$$
Then, $\Curl\TS=\Curl\T$ and
$\TS\in\HCom\cap\HDczom\cap(\harmdi^N)^{\bot}$.
Again, by Corollary \ref{poincaremaxten} we have \eqref{estpsi}.
Note that
$$\Tskew_{\TR}=\pi_{\soN}\TR=\skew\oint_{\om}\TR\dl\in\soN\subset\HCzom.$$
As before, by orthogonality, Lemma \ref{genkornlem} (ii') applied to $\TR$ and \eqref{estpsi}
\begin{align*}
\normLtom{\T-\Tskew_{\TR}}^2
=\normLtom{\TR-\Tskew_{\TR}}^2+\normLtom{\TS}^2
&\leq\ck^2\normLtom{\sym\TR}^2+\normLtom{\TS}^2\\
&\leq2\ck^2\normLtom{\sym\T}^2+(1+2\ck^2)\normLtom{\TS}^2\\
&\leq\cone^2\big(\normLtom{\sym\T}^2+\normLtom{\Curl\T}^2\big).
\end{align*}
For $\TS=\Curl^{*}X$ with 
$X\in\Hgen{}{}{\circ}(\Curl^{*};\om)=\HDczom\cap(\harmdi^N)^{\bot}$,
where $\Curl^{*}\cong-\delta_{2}$ denotes the formal adjoint 
of $\Curl\cong\ed_{1}$, and all $\Tskew\in\soN$ we have
$$\scpLtom{\Tskew_{\TS}}{\Tskew}
=\scps{\int_{\om}\TS\dl}{\Tskew}_{\rNtN}
=\scpLtom{\Curl^{*}X}{\Tskew}
=\scpLtom{X}{\Curl\Tskew}=0,$$
which shows $\Tskew_{\TS}=0$ by setting $\Tskew:=\Tskew_{\TS}$.
Hence $\Tskew_{\T}=\Tskew_{\TR}$.
The proof of (ii') is complete, since all other remaining assertions are trivial.
Finally, to show (ii), we follow the proof of (ii') 
up to the point, where $\Tskew_{\TR}$ was introduced.
Now, by Lemma \ref{genkornlem} (ii) for $\TR$ we get a piece-wise constant 
skew-symmetric tensor $\Tskew:=\Tskew_{\TR}$.
We note that in general $\Tskew$ does not belong to $\HCom$ anymore.
Hence, we loose the $\Ltom$-orthogonality $\TR-\Tskew\,\bot\,\TS$.
But again, by Lemma \ref{genkornlem} (ii) and \eqref{estpsi}
\begin{align}
\normLtom{\T-\Tskew}
&\leq\normLtom{\TR-\Tskew}+\normLtom{\TS}
\leq\ck\normLtom{\sym\TR}+\normLtom{\TS}\nonumber\\
&\leq\ck\normLtom{\sym\T}+(1+\ck)\normLtom{\TS}\nonumber\\
&\leq\ck\normLtom{\sym\T}+(1+\ck)\cm\normLtom{\Curl\T}\nonumber\\
&\leq\ctwo\big(\normLtom{\sym\T}^2+\normLtom{\Curl\T}^2\big)^{1/2}\nonumber
\intertext{with}
\mylabel{explicitconsttwo}
\ctwo&:=\sqrt{2}\max\{\ck,\cm(1+\ck)\},
\end{align}
which proves (ii).
\end{proofof}

\section{One Additional Result}

As in \cite[sections 3.4]{neffpaulywitschgenkornrtsli} 
we can prove a generalization for media with structural changes.
To apply the main result from \cite{pompekorn},
let $\mu\in\Czomb$ be a $(N\times N)$-matrix field
satisfying $\det\mu\geq\hat{\mu}>0$. 

\begin{cor}
\mylabel{corothermedia}
Let $\gat\neq\emptyset$ and let the pair $\pair$ be admissible.
Then there exists $c>0$ such that
$$c\normLtom{\T}\leq\normLtom{\sym(\mu\T)}+\normLtom{\Curl\T}$$
holds for all tensor fields $\T\in\HCcgatom$. 
In other words, on $\HCcgatom$ the right hand side 
defines a norm equivalent to the standard norm in $\HCom$.
\end{cor}

\appendix

\section{Construction of Hodge-Helmholtz Projections}

We want to point out how to compute the projections
in the Hodge-Helmholtz decompositions in Lemma \ref{hodgehelmdeco}.
Recalling from Lemma \ref{hodgehelmdeco} the orthogonal decompositions
\begin{align*}
\Ltqom
&=\ed\Dqmocgatom\oplus\Deqczganom\\
&=\Dqczgatom\oplus\cd\Deqpocganom\\
&=\ed\Dqmocgatom\oplus\harmdiq\oplus\cd\Deqpocganom
\end{align*}
we denote the corresponding $\Ltqom$-orthogonal projections by
$\pi_{\ed}$, $\pi_{\cd}$ and $\pi_{\dirichlet}$.
Then, we have $\pi_{\dirichlet}=\id-\pi_{\ed}-\pi_{\cd}$ and
\begin{align*}
\pi_{\ed}\Ltqom
&=\ed\Dqmocgatom
=\ed\xqmoom,&
\xqmoom&:=\Dqmocgatom\cap\cd\Deqcganom,\\
\pi_{\cd}\Ltqom
&=\cd\Deqpocganom
=\cd\yqpoom,&
\yqpoom&:=\Deqpocganom\cap\ed\Dqcgatom,\\
\pi_{\dirichlet}\Ltqom&=\harmdiq.
\end{align*}
By Poincar\'e's estimate, i.e., Lemma \ref{poincarediff},
we have
\begin{align}
\forall\,E&\in\xqmoom&
\normLtqmoom{E}&\leq\cpqmo\normLtqom{\ed E},\mylabel{poincarespecd}\\
\forall\,H&\in\yqpoom&
\normLtqpoom{H}&\leq\cpqpo\normLtqom{\cd H}.\mylabel{poincarespecde}
\end{align}
Hence, the bilinear forms
\begin{align*}
(\tilde{E},E)&\mapsto\scpLtqom{\ed\tilde{E}}{\ed E},&
(\tilde{H},H)&\mapsto\scpLtqom{\cd\tilde{H}}{\cd H}
\intertext{are continuous and coercive over 
$\xqmoom$ and $\yqpoom$, respectively.
Moreover, for any $F\in\Ltqom$ the linear functionals}
E&\mapsto\scpLtqom{F}{\ed E},&
H&\mapsto\scpLtqom{F}{\cd H}
\end{align*}
are continuous over $\xqmoom$ respectively $\yqpoom$.
Thus, by Lax-Milgram's theorem we get
unique solutions $E_{\ed}\in\xqmoom$ and $H_{\cd}\in\yqpoom$ 
of the two variational problems
\begin{align}
\scpLtqom{\ed E_{\ed}}{\ed E}&=\scpLtqom{F}{\ed E}&
&\forall\,E\in\xqmoom,\mylabel{laxmilgramone}\\
\scpLtqom{\cd H_{\cd}}{\cd H}&=\scpLtqom{F}{\cd H}&
&\forall\,H\in\yqpoom\mylabel{laxmilgramtwo}
\end{align}
and the corresponding solution operators,
mapping $F$ to $E_{\ed}$ and $H_{\cd}$, respectively,
are continuous. In fact, we have as usual
$$\normLtqom{\ed E_{\ed}}\leq\normLtqom{F},\quad
\normLtqom{\cd H_{\cd}}\leq\normLtqom{F},$$
respectively, and therefore together 
with \eqref{poincarespecd} and \eqref{poincarespecde}
\begin{align*}
\norm{E_{\ed}}_{\xqmoom}&=\norm{E_{\ed}}_{\Dqmoom}
\leq\sqrt{1+\cpqmo^2}\normLtqom{F},\\
\norm{H_{\cd}}_{\yqpoom}&=\norm{H_{\cd}}_{\Deqpoom}
\leq\sqrt{1+\cpqpo^2}\normLtqom{F}.
\end{align*}
Since $\ed\Dqmocgatom=\ed\xqmoom$
and $\cd\Deqpocganom=\cd\yqpoom$
we see that \eqref{laxmilgramone} and \eqref{laxmilgramtwo}
hold also for $E\in\Dqmocgatom$ and
$H\in\Deqpocganom$, respectively,
and that
\begin{align*}
F-\ed E_{\ed}&\in\big(\ed\xqmoom\big)^{\bot}
=\big(\ed\Dqmocgatom\big)^{\bot}=\Deqczganom,\\
F-\cd H_{\cd}&\in\big(\cd\yqpoom\big)^{\bot}
=\big(\cd\Deqpocganom\big)^{\bot}=\Dqczgatom.
\end{align*}
Hence, we have found our projections since
\begin{align*}
\pi_{\ed}F&:=\ed E_{\ed}\in\ed\xqmoom\subset\Dqczgatom,\\
\pi_{\cd}F&:=\cd H_{\cd}\in\cd\yqpoom\subset\Deqczganom
\end{align*}
and
$$\pi_{\dirichlet}F:=F-\ed E_{\ed}-\cd H_{\cd}
\in\Dqczgatom\cap\Deqczganom=\harmdiq.$$

Explicit formulas for the dimensions of $\harmdiq$
or explicit constructions of bases of $\harmdiq$ 
depending on the topology of the pair $\pair$ can be found, e.g.,
in \cite{picardboundaryelectro} for the case $\gat=\ga$ or $\gat=\emptyset$,
or in \cite{goldshteinmitreairinamariushodgedecomixedbc} for the general case.

\begin{acknow}
We heartily thank Kostas Pamfilos for the beautiful pictures of 3D sliceable domains.
\end{acknow}

\bibliographystyle{plain} 
\bibliography{/Users/paule/Library/texmf/tex/TeXinput/bibtex/paule}

\begin{thebibliography}{10}

\bibitem{agmonbook}
S.~Agmon.
\newblock {\em Lectures on elliptic boundary value problems}.
\newblock Van Nostrand, New York, London, Toronto, 1965.

\bibitem{brownmixedproblaplacelipschitz}
R.~Brown.
\newblock The mixed problem for {L}aplace's equation in a class of {L}ipschitz
  domains.
\newblock {\em Comm. Partial Differential Equations}, 19(7-8):1217--1233, 1994.

\bibitem{costabelremmaxlip}
M.~Costabel.
\newblock A remark on the regularity of solutions of {M}axwell's equations on
  {L}ipschitz domains.
\newblock {\em Math. Methods Appl. Sci.}, 12(4):365--368, 1990.

\bibitem{goldshteinmitreairinamariushodgedecomixedbc}
V.~Gol'dshtein, I.~Mitrea, and M.~Mitrea.
\newblock {H}odge decompositions with mixed boundary conditions and
  applications to partial differential equations on {L}ipschitz manifolds.
\newblock {\em J. Math. Sci. (N.Y.)}, 172(3):347--400, 2011.

\bibitem{jakabmitreairinamariusfinensolhodgedeco}
T.~Jakab, I.~Mitrea, and M.~Mitrea.
\newblock On the regularity of differential forms satisfying mixed boundary
  conditions in a class of {L}ipschitz domains.
\newblock {\em Indiana Univ. Math. J.}, 58(5):2043--2071, 2009.

\bibitem{jochmanncompembmaxmixbc}
F.~Jochmann.
\newblock A compactness result for vector fields with divergence and curl in
  ${L}^q({\Omega})$ involving mixed boundary conditions.
\newblock {\em Appl. Anal.}, 66:189--203, 1997.

\bibitem{kuhndiss}
P.~Kuhn.
\newblock {\em Die {M}axwellgleichung mit wechselnden {R}andbedingungen}.
\newblock Dissertation, Universit\"at Essen, Fachbereich Mathematik,
  http://arxiv.org/abs/1108.2028, {\it Shaker}, 1999.

\bibitem{kuhnpaulyregmax}
P.~Kuhn and D.~Pauly.
\newblock Regularity results for generalized electro-magnetic problems.
\newblock {\em Analysis (Munich)}, 30(3):225--252, 2010.

\bibitem{leistheoem}
R.~Leis.
\newblock {Z}ur {T}heorie elektromagnetischer {S}chwingungen in anisotropen
  inhomogenen {M}edien.
\newblock {\em Math. Z.}, 106:213--224, 1968.

\bibitem{leistheomaxgmd}
R.~Leis.
\newblock {Z}ur {T}heorie der zeitunabh\"angigen {M}axwellschen {G}leichungen.
\newblock {\em {B}erichte der {G}esellschaft f\"ur {M}athematik und
  {D}atenverarbeitung}, 50, 1971.

\bibitem{leisbook}
R.~Leis.
\newblock {\em Initial Boundary Value Problems in Mathematical Physics}.
\newblock Teubner, Stuttgart, 1986.

\bibitem{neffpaulywitschgenkornrt}
P.~Neff, D.~Pauly, and K.-J. Witsch.
\newblock A canonical extension of {K}orn's first inequality to {H}({C}url)
  motivated by gradient plasticity with plastic spin.
\newblock {\em C. R. Acad. Sci. Paris, Ser. I}, 349:1251--1254, 2011.

\bibitem{neffpaulywitschgenkornpamm}
P.~Neff, D.~Pauly, and K.-J. Witsch.
\newblock A {K}orn's inequality for incompatible tensor fields.
\newblock {\em Proc. Appl. Math. Mech. (PAMM)}, 11:683--684, 2011.

\bibitem{neffpaulywitschgenkornrn}
P.~Neff, D.~Pauly, and K.-J. Witsch.
\newblock {M}axwell meets {K}orn: A new coercive inequality for tensor fields
  in $\mathbb{R}^{N\times N}$ with square-integrable exterior derivative.
\newblock {\em Math. Methods Appl. Sci.}, 35:65--71, 2012.

\bibitem{neffpaulywitschgenkornrtzap}
P.~Neff, D.~Pauly, and K.-J. Witsch.
\newblock On a canonical extension of {K}orn's first and {P}oincar\'e's
  inequality to {H}({C}url).
\newblock {\em J. Math. Sci. (N.Y.)}, 185(1):64--70, 2012.

\bibitem{neffpaulywitschgenkornrtsli}
P.~Neff, D.~Pauly, and K.-J. Witsch.
\newblock {P}oincar\'e meets {K}orn via {M}axwell: {E}xtending {K}orn's first
  inequality to incompatible tensor fields.
\newblock {\em submitted}, ?:?--?, 2012.

\bibitem{paulytimeharm}
D.~Pauly.
\newblock Low frequency asymptotics for time-harmonic generalized {M}axwell
  equations in nonsmooth exterior domains.
\newblock {\em Adv. Math. Sci. Appl.}, 16(2):591--622, 2006.

\bibitem{paulystatic}
D.~Pauly.
\newblock Generalized electro-magneto statics in nonsmooth exterior domains.
\newblock {\em Analysis (Munich)}, 27(4):425--464, 2007.

\bibitem{paulyasym}
D.~Pauly.
\newblock Complete low frequency asymptotics for time-harmonic generalized
  {M}axwell equations in nonsmooth exterior domains.
\newblock {\em Asymptot. Anal.}, 60(3-4):125--184, 2008.

\bibitem{paulydeco}
D.~Pauly.
\newblock {H}odge-{H}elmholtz decompositions of weighted {S}obolev spaces in
  irregular exterior domains with inhomogeneous and anisotropic media.
\newblock {\em Math. Methods Appl. Sci.}, 31:1509--1543, 2008.

\bibitem{picardpotential}
R.~Picard.
\newblock {R}andwertaufgaben der verallgemeinerten {P}otentialtheorie.
\newblock {\em Math. Methods Appl. Sci.}, 3:218--228, 1981.

\bibitem{picardboundaryelectro}
R.~Picard.
\newblock On the boundary value problems of electro- and magnetostatics.
\newblock {\em Proc. Roy. Soc. Edinburgh Sect. A}, 92:165--174, 1982.

\bibitem{picardcomimb}
R.~Picard.
\newblock An elementary proof for a compact imbedding result in generalized
  electromagnetic theory.
\newblock {\em Math. Z.}, 187:151--164, 1984.

\bibitem{picardlowfreqmax}
R.~Picard.
\newblock On the low frequency asymptotics in and electromagnetic theory.
\newblock {\em J. Reine Angew. Math.}, 354:50--73, 1984.

\bibitem{picarddeco}
R.~Picard.
\newblock Some decomposition theorems and their applications to non-linear
  potential theory and {H}odge theory.
\newblock {\em Math. Methods Appl. Sci.}, 12:35--53, 1990.

\bibitem{picardweckwitschxmas}
R.~Picard, N.~Weck, and K.-J. Witsch.
\newblock Time-harmonic {M}axwell equations in the exterior of perfectly
  conducting, irregular obstacles.
\newblock {\em Analysis (Munich)}, 21:231--263, 2001.

\bibitem{pompekorn}
W.~Pompe.
\newblock Korn{'}s first inequality with variable coefficients and its
  generalizations.
\newblock {\em Comment. Math. Univ. Carolinae}, 44(1):57--70, 2003.

\bibitem{saranenmaxkegel}
J.~Saranen.
\newblock {\"U}ber das {V}erhalten der {L}\"osungen der {M}axwellschen
  {R}andwertaufgabe in {G}ebieten mit {K}egelspitzen.
\newblock {\em Math. Methods Appl. Sci.}, 2(2):235--250, 1980.

\bibitem{saranenmaxnichtglatt}
J.~Saranen.
\newblock {\"U}ber das {V}erhalten der {L}\"osungen der {M}axwellschen
  {R}andwertaufgabe in einigen nichtglatten {G}ebieten.
\newblock {\em Ann. Acad. Sci. Fenn. Ser. A I Math.}, 6(1):15--28, 1981.

\bibitem{webercompmax}
C.~Weber.
\newblock A local compactness theorem for {M}axwell's equations.
\newblock {\em Math. Methods Appl. Sci.}, 2:12--25, 1980.

\bibitem{weckmax}
N.~Weck.
\newblock {M}axwell's boundary value problems on {R}iemannian manifolds with
  nonsmooth boundaries.
\newblock {\em J. Math. Anal. Appl.}, 46:410--437, 1974.

\bibitem{witschremmax}
K.-J. Witsch.
\newblock A remark on a compactness result in electromagnetic theory.
\newblock {\em Math. Methods Appl. Sci.}, 16:123--129, 1993.

\bibitem{wlokabook}
J.~Wloka.
\newblock {\em Partial Differential Equations}.
\newblock Teubner, Stuttgart, 1982.

\end{thebibliography}

\end{document}